\documentclass[final,onefignum,onetabnum]{siamart220329}

\usepackage{braket,amsfonts}
\usepackage{lineno,todonotes}
\usepackage{array}

\usepackage[caption=false]{subfig}
\usepackage{pgfplots}

\newsiamthm{claim}{Claim}
\newsiamremark{remark}{Remark}
\newsiamremark{hypothesis}{Hypothesis}
\newsiamremark{assumption}{Assumption}
\crefname{hypothesis}{Hypothesis}{Hypotheses}

\usepackage{algorithmic}

\usepackage{graphicx,epstopdf}

\Crefname{ALC@unique}{Line}{Lines}

\usepackage{amsopn}

\usepackage{xspace}
\usepackage{bold-extra}
\usepackage[most]{tcolorbox}

\colorlet{texcscolor}{blue!50!black}
\colorlet{texemcolor}{red!70!black}
\colorlet{texpreamble}{red!70!black}
\colorlet{codebackground}{black!25!white!25}

\lstdefinestyle{siamlatex}{%
  style=tcblatex,
  texcsstyle=*\color{texcscolor},
  texcsstyle=[2]\color{texemcolor},
  keywordstyle=[2]\color{texemcolor},
  moretexcs={cref,Cref,maketitle,mathcal,text,headers,email,url},
}

\tcbset{%
  colframe=black!75!white!75,
  coltitle=white,
  colback=codebackground, 
  colbacklower=white, 
  fonttitle=\bfseries,
  arc=0pt,outer arc=0pt,
  top=1pt,bottom=1pt,left=1mm,right=1mm,middle=1mm,boxsep=1mm,
  leftrule=0.3mm,rightrule=0.3mm,toprule=0.3mm,bottomrule=0.3mm,
  listing options={style=siamlatex}
}

\newtcblisting[use counter=example]{example}[2][]{%
  title={Example~\thetcbcounter: #2},#1}

\newtcbinputlisting[use counter=example]{\examplefile}[3][]{%
  title={Example~\thetcbcounter: #2},listing file={#3},#1}

\DeclareTotalTCBox{\code}{ v O{} }
{ 
  fontupper=\ttfamily\color{black},
  nobeforeafter,
  tcbox raise base,
  colback=codebackground,colframe=white,
  top=0pt,bottom=0pt,left=0mm,right=0mm,
  leftrule=0pt,rightrule=0pt,toprule=0mm,bottomrule=0mm,
  boxsep=0.5mm,
  #2}{#1}

\patchcmd\newpage{\vfil}{}{}{}
\renewcommand{\d}{\ensuremath{\mathrm{d}}}

\usepackage{enumitem}

\begin{tcbverbatimwrite}{tmp_\jobname_header.tex}
\title{Probability Flow Approach to the Onsager--Machlup Functional for Jump-Diffusion Processes
\thanks{
Submitted to the editors DATE.
\funding{XZ acknowledges the support from 
 Hong Kong General Research Funds  (11308121, 11318522,   11308323),  and  the NSFC/RGC Joint Research Scheme [RGC Project No. N-CityU102/20 and NSFC Project No.
12061160462]. JQD acknowledges the support from the NSFC grant 12141107 and  the Guangdong-Dongguan Joint Research Grant 2023A151514 0016.}}}

\author{Yuanfei Huang\thanks{Department of Data Science, City University of Hong Kong, Kowloon, Hong Kong SAR (\email{yhuan26@cityu.edu.hk}).}
\and Xiang Zhou\thanks{Department of Data Science and Department of Mathematics, City University of Hong Kong, Kowloon, Hong Kong SAR (Corresponding Author. \email{ xizhou@cityu.edu.hk}). }
\and Jinqiao Duan\thanks{Department of Mathematics and Department of Physics, Great Bay University, Dongguan, Guangdong 523000, China. (\email{duan@gbu.edu.cn})}
 }

\headers{Onsager--Machlup functional for jump-diffusion}{Yuanfei Huang, Xiang Zhou and Jinqiao Duan}
\end{tcbverbatimwrite}
\input{tmp_\jobname_header.tex}

\ifpdf
\hypersetup{ pdftitle={Guide to Using  SIAM'S \LaTeX\ Style} }
\fi

\begin{document}
\maketitle

\begin{tcbverbatimwrite}{tmp_\jobname_abstract.tex}
\begin{abstract}
The Onsager--Machlup action functional is an important concept in statistical mechanics and thermodynamics to describe the probability of fluctuations in nonequilibrium systems. It provides a powerful tool for analyzing and predicting the behavior of complex stochastic systems.
For    diffusion process, 
the path integral method and the Girsanov transformation are two main approaches to construct the Onsager--Machlup functional.
 However, it is a long-standing challenge to  apply these two methods to      the jump-diffusion process, because  the complexity of jump noise presents intrinsic   technical barriers to derive the Onsager--Machlup functional. In this work, we propose a new strategy to solve this problem by   utilizing the equivalent probabilistic flow between the pure diffusion process and the jump-diffusion process. For the first time, we rigorously establish the closed-form expression of the Onsager--Machlup functional for jump-diffusion processes with finite jump activity, which include an important term of the L\'{e}vy intensity at the origin. The same probability flow approach is further applied to    the  L\'{e}vy process with infinite jump activity, 
  and yields a time-discrete version of the   Onsager--Machlup functional.
\end{abstract}


\end{tcbverbatimwrite}
\input{tmp_\jobname_abstract.tex}

\section{Background and Introduction}
\label{sec:intro}

\subsection{The classical Onsager--Machlup action functional}

The most likely path in a system affected by noise can be well  described by the minimizer of  the \emph{Onsager--Machlup action functional} (OM functional), which is a famous result in the theory of thermodynamic fluctuations, and is frequently employed in the meta-stable systems  to study transition mechanisms.  
This variational principle based on Onsager--Machlup as well as other action functionals \cite{FW2012} has stimulated many interesting modelling, computation  and application  works,  such as path integrals \cite{Kath1981path,Wio2013}, minimum action paths \cite{weinan2004minimum,zhou2008adaptive,lin2019quasi,zhou2010study,wan2010study,ge2012analytical,du2021graph}, trajectory sampling \cite{KA20,Lu2017}, transition-path theory \cite{EVE10, GLLL23}, transition rates \cite{Zuckerman2000, Gobbo2012}, biomolecular models \cite{Faccioli2006, Wang2006,QianGe2021Book}, information projection \cite{selk2021information, bierkens2014explicit} and Bayesian inverse problems \cite{dashti2013map, AKLS21}.

In 1953, when Onsager and Machlup \cite{OM53, MO53} examined the probabilities of thermodynamic state trajectories in irreversible processes, they first introduced the OM functional for Gaussian systems by using Feynman's path integral to express the transition probability. Their work was restricted to It\^{o} stochastic differential equations (SDEs) with  constant diffusion coefficients and linear drifts. Tisza and Manning \cite{tisza1957fluctuations}  expanded the OM functional to SDEs that have nonlinear drift terms. The mathematical progress in the works of 
\cite{Durr1978} and \cite{Ikeda1980} 
rigorously established the expression of the OM functional for reversible and general It\^{o} diffusion SDEs, respectively, by investigating  the probability that a trajectory is inside   the tube along a  differentiable curve. They interpreted the OM functional as the action in the infinitesimal limit of the tube size.

The   OM functional  is well-known for   the   following It\^{o}  diffusion  SDE in $\mathbb{R}^d$,
\begin{equation}
    \left\{\begin{aligned} 
    \label{SDE}
        \d X_t=&\ b(X_{t})\d t +  \sigma\d B_t,\quad t>0,\\
         X_0=&\ x_0\in\mathbb{R}^d,
    \end{aligned}\right.
\end{equation}
where $b$ is a measurable vector function, $\sigma$ is a constant, and  $B$ is a standard Brownian motion in $\mathbb{R}^d$.  
For  \eqref{SDE}, the OM functional in a   time interval $[0,T]$ is given by 
{\begin{equation} \label{eq:OM4SDE}
    \begin{aligned}
        S_{X}^{\mathrm{OM}}(\psi,\dot{\psi})=&\ \frac{1}{2}\int_0^T \left[ \frac{|\dot{\psi}_s - b(\psi_s)|^2}{\sigma^2} + \nabla\cdot b(\psi_s) \right]\d s.
\end{aligned}
\end{equation}}There are two main methodologies   to 
  derive  this Onsager--Machlup functional \eqref{eq:OM4SDE}: the path integral approach \cite{Kath1981path,Wio2013}
  and the change of measure via Girsanov theorem \cite{Durr1978, Ikeda1980}. 
  
  The path integral sums the likelihood over all possible paths, each   weighted by the OM action. The finite-dimensional distributions are first derived  in  this approach before taking the   limit   of  time discretization.  It expresses
  the transition probability  from $x_0$ at $t=0$ to $x_T$ at $t=T$ in the form of Feynman's path integral \cite{Kath1981path,Wio2013},  
{\begin{equation} \label{eq:PISDE}
    \begin{aligned}
        p_T(x_0,x_T)=&\ \lim_{n\to\infty}\int_{\prod_{i=1}^{n-1}\mathbb{R}^d}p_{t_{1}-t_{0}}(x_{0},x_{1})\cdots p_{t_{n}-t_{n-1}}(x_{n-1},x_{T})\d x_1\cdots\d x_{n-1}\\
        =&\ \int_{\psi\in C_{x_0,x_T}[0,T]}\exp\left(-S_X^{\mathrm{OM}}(\psi,\dot{\psi})\right)\mathcal{D}(\psi).
        \end{aligned}
\end{equation}
}Here $C_{x_0,x_T}[0,T]$ denotes the space of continuous paths with initial point $x_0$ and terminal point $x_T$, and $\mathcal{D}$ is the corresponding path measure on $C_{x_0,x_T}[0,T]$.

The second popular approach is to evaluate the probability that the stochastic process $X_t$ inside a $\delta$-width tube around a smooth path and then to work on the limit of vanishing $\delta$. The interpretation of the OM functional is that for   a sufficiently small $\delta>0$, there is a    constant $C(\delta,T)>0$ depending on $\delta$ and $T$, such that  
{\begin{equation} \label{eq:Pob}
\begin{split}
\mathbb{P}\left(\sup_{t\in[0,T]}|X_t-\psi_t|\leq\delta\right) \approx C(\delta,T)\exp(-S_{X}^{\mathrm{OM}}(\psi,\dot{\psi})),  
\end{split}
\end{equation}
}holds for any $\psi\in C^2([0,T],\mathbb R^d)$ satisfying $\psi_0=x_0$; refer to \cite{Durr1978, Ikeda1980}
To calculate the probability on the left-hand side of \eqref{eq:Pob}, it is necessary to apply the change of measure through the Girsanov theorem. The central point  is to
identify the correct exponent term as $\delta\downarrow 0$.



\subsection{ Onsager--Machlup functional for non-Gaussian systems}
Even though diffusion processes are widely used, many complex biological and physical systems are driven by noise of a non-Gaussian type. The corresponding OM functional for these systems is still unknown, though. There aren't many related works for the jump diffusion process, even when it comes to the solution's asymptotic behaviors and the most likely transition path.
\cite{bardina2002asymptotic} only addressed the asymptotic evaluation of the probability for a tube around the jump path for linear stochastic differential equations driven by a Poisson process. 
With symmetric $\alpha$-stable L\'{e}vy motion, \cite{huang2019characterization} described the most likely transition path of SDE in terms of a deterministic dynamical system, but left the question of the OM functional unresolved.

Applying the well-known path integral and Girsanov theorem to deriving the OM action functional for jump diffusion process presents a few fundamental difficulties.  
A typical jump-diffusion process's transition density in \eqref{eq:PISDE} becomes non-trivially complex for the path integral approach, failing to provide a tractable functional representation of paths as in the diffusion process. 
Accurately accounting for the jumps within the $\delta$-width tube in the limit of $\delta$ tending to 0 is essential for the second attempt of using change of measure via Girsanov theorem; however, the summation of such jumps in the tube is unbounded, which presents a significant challenge for the limit of vanishing $\delta$.  \cite{chao2019onsager} has attempted to explore this second approach to the one dimensional jump diffusion process by applying \cite{Durr1978} to it. Regretfully, this outcome does not ensure a full form of the corresponding OM functional since it does not address the subtle problem of unbounded number of jumps in the tube.

\subsection{Main Contributions}
The Onsager--Machlup functional holds a crucial role in the analysis of stochastic dynamical systems. As we have previously introduced, in the realm of Gaussian systems, this functional has been extensively and effectively employed across a wide range of disciplines. However, as highlighted in earlier sections, traditional approaches to deriving the Onsager--Machlup functional encounter substantial technical difficulties when applied to non-Gaussian contexts. In this paper, we tackle this significant theoretical gap by introducing a novel methodology. This innovative approach capitalizes on the newly established concept of probability flow equivalence between jump-diffusion and diffusion processes. By employing this method, we successfully derive the closed-form Onsager--Machlup functional for jump-diffusion processes characterized by finite jump activity. Furthermore, for jump-diffusions with infinite jump activity, our approach provides a discrete version of the corresponding Onsager--Machlup functional.

Let us introduce our results in more detail. Let $(\Omega,\mathcal{F},\mathbb{P})$ be a probability space equipped with a filtration $\{\mathcal{F}_t,t\geq0\}$ that satisfies the usual hypotheses, i.e., $\mathcal{F}_0$ contains all sets of $\mathbb{P}$-measure zero, and $\mathcal{F}_t=\mathcal{F}_{t+}$ where $\mathcal{F}_{t+}=\cap_{\varepsilon>0}\mathcal{F}_{t+\varepsilon}$.
Consider the following jump-diffusion stochastic differential equation with finite jump activity:
\begin{equation}\label{sdeBL:finitepoisson}
    \left\{\begin{aligned}
        \d X_t=&\ b(X_{t-})\d t + \sigma \d B_t +  J_t\d N_t,\quad t>0,\\
         X_0=&\ x_0\in\mathbb{R}^d,
    \end{aligned}\right.
\end{equation}
where $N$, being independent of the Brownian motion $B$, is an Poisson process with state-dependent stochastic intensity $\lambda(X_{t-})$ and jump size of 1.  The term  $J_t$ represents the random jump size with intensity $\nu_J$, which is independent of both $B$ and $N$. Our results, as presented in Theorem \ref{theo:PI}, demonstrate that the Onsager--Machlup functional for \eqref{sdeBL:finitepoisson} is given by
{\small\begin{equation}\label{OM:jump}
    \begin{split}
        S_{X}^{\mathrm{OM}}(\psi,\dot{\psi})=&\ \frac{1}{2} \int_{0}^{T} \Bigg[ \frac{1}{\sigma^2}\left|\dot{\psi}_s -  b(\psi_s) - \ell_J(\psi_s) \right|^2   +  \nabla\cdot\left( b(\psi_s) +  \ell_J(\psi_s)\right) + \tilde{\ell}_J(\psi_s) \Bigg]\d s ,
    \end{split}
\end{equation}
}where 
{\small\begin{equation}\label{lJ}
\begin{aligned}
    \ell_J(x):=&\ \int_{\mathbb{R}^d} z\int_0^1\frac{\lambda(x-\theta z)\nu_J( -\theta z)}{\nu_J(0)}\d\theta\ \nu_J(\d z),\\
    \tilde{\ell}_J(x):=&\ \int_{\mathbb{R}^d}z\cdot\int_0^1\lambda(x-\theta z)\frac{\nabla\nu_J(-\theta z)\nu_J(0)-\nu_J(-\theta z)\nabla\nu_J(0)}{(\nu_J(0))^2}\d\theta\ \nu_J(\d z).
\end{aligned}
\end{equation}}


Our second result is about the jump diffusion processes with infinite jump activity. Consider a scalar Markovian jump-diffusion process $X$ whose generator is given by:
{\small\begin{equation}\label{sdeBL2}
    \begin{aligned}
         \mathcal{L}f(x)=&\ b(x)f'(x) + \frac{\sigma^2}{2}f''(x) + \int_{\mathbb{R}\backslash\{0\}} \left( f(x + F(x,z) ) -  f(x) - 1_{\{|z|<1\}}F(x,z)f'(x) \right) \nu(x,z)\d z,
   \end{aligned} 
\end{equation}
}where the mappings $b:\mathbb{R}\to\mathbb{R}$ and $F:\mathbb{R}\times\mathbb{R}\to\mathbb{R}$ are all assumed to be measurable, $\sigma$ is a positive constant, $\nu: \mathbb{R} \times \mathbb{R} \backslash \{0\} \to (0, \infty)$ is a L\'{e}vy measure being absolutely continuous respect to the Lebesgue measure. For any $n\in\mathbb{N}$ and Borel sets $\{A_i\subset\mathbb{R}\}_{i=1}^n$, our result in {Theorem \ref{theo:infinite}} claims that,
{\small \begin{equation*} 
    \begin{aligned}
       & \mathbb{P}(X_{t_i}\in A_i,\ i=1,\cdots,n) \\
       = &\ \int_{\prod_{i=1}^{n-1}A_i}  \exp\left\{ - S_X^{\mathrm{dOM}}(\{x_i\}_{i=0}^n) + \mathcal{O}(\Delta t) \right\} \left(\frac{1}{\sqrt{(2\pi)\Delta t}}\right)^n\d x_1\cdots\d x_{n-1},
    \end{aligned}
\end{equation*}
}where the discrete version of Onsager--Machlup function reads
{\small \begin{equation}\label{dOM}
    \begin{aligned}
        & S_X^{\mathrm{dOM}}(\{x_i\}_{i=0}^n)=\ \sum_{i=1}^n   \frac{1}{2\sigma^2}\left| \frac{x_{i}-x_{i-1}}{\Delta t} - b(x_i) + \int_{|z|<1}F(x_i,z)\nu(x_i,\d z) \right. \\
        & \left. - \int_0^1\int_{\mathbb{R}\backslash\{0\}} F(\mathcal{T}^{-1}_{F,\theta,z}(x_i),z )|\mathcal{J}_{F,\theta,z}(x_i)|\frac{\nu_F(x_{i-1},\mathcal{T}^{-1}_{F,\theta,z}(x_i) -x_{i-1})}{\nu_F(x_{i-1},x_i-x_{i-1})} \right. \\
        & \left. \times \nu(\mathcal{T}^{-1}_{F,\theta,z}(x_i),\d z)\ \d\theta  \right|^2 \Delta t  + \sum_{i=1}^n\Delta t \frac{\d}{\d x}\left(b(x) - \int_{|z|<1}F(x,z)\nu(x,\d z)  \right. \\
        & \left.\left.  + \int_0^1\int_{\mathbb{R}\backslash\{0\}} F(\mathcal{T}^{-1}_{F,\theta,z}(x),z )|\mathcal{J}_{F,\theta,z}(x)|\frac{\nu_F(x_{i-1},\mathcal{T}^{-1}_{F,\theta,z}(x) -x_{i-1})}{\nu_F(x_{i-1},x-x_{i-1})} \nu(\mathcal{T}^{-1}_{F,\theta,z}(x),\d z)\ \d\theta \right)\bigg|_{x=x_i} ,\right.
    \end{aligned}
\end{equation}
}where
$\mathcal{T}_{F,\theta,z}: x\mapsto x + \theta F(x,z)$ is assumed to be invertible and $\mathcal{J}_{F,\theta,z}$ is the Jaccobian matrix of $\mathcal{T}^{-1}_{F,\theta,z}$, and $\nu_F(y,x-y):=-\frac{\partial}{\partial (x-y)}\left(\int_{\{z : F(y, z) \geq  x-y\}}\nu(y,z)\d z\right)$ for $x\neq y$.

Let us briefly discuss our primary results. Unlike the pure diffusion described by \eqref{SDE}, the additional terms $\ell_J$ and $\tilde{\ell}$ in \eqref{OM:jump} arise from jump noise $J_t \d N_t$. If the system is driven only by Brownian noise as in \eqref{sdeBL:finitepoisson}, the Onsager--Machlup functional \eqref{OM:jump} aligns with the standard one \eqref{eq:OM4SDE}. The influence of the jump noise is largely determined by the properties of the L\'{e}vy measure, especially its behaviour at zero. In systems with infinite jump activity, the L\'{e}vy measure's definition excludes zero, preventing a continuous functional form for \eqref{dOM}. However, for system \eqref{sdeBL2} with $ F(x, z) = z $ and $ \nu(x, z) = \lambda(x)\nu_J(z) $ (i.e., it turns to the finite jump activity case), the discrete OM functional \eqref{dOM} converges to the continuous form \eqref{OM:jump}.

The structure of the paper is as follows: In Section \ref{FPEequi}, we derive the L\'{e}vy--Fokker--Planck equations and establish the probability flow equivalence for jump-diffusion processes with finite jump activity. Based on this equivalence, Section \ref{OMderivation} presents the derivation of the Onsager--Machlup functional for these processes using path integrals and tube-probability estimates. In Section \ref{sec:infinite}, we extend our discussion to jump-diffusions with infinite jump activity. Finally, in Section \ref{discussion}, we summarise our results and discuss the contributions of this paper.

\section{Jump-Diffusion Processes with Finite Jump Activity}\label{FPEequi}

In this section, we consider the jump diffusion process with finite jump activity described in \eqref{sdeBL:finitepoisson}. Based on the corresponding L\'{e}vy--Fokker--Planck equation, we demonstrate the probability flow equivalence to a diffusion process. We make the following assumptions on the jump-diffusion SDE \eqref{sdeBL:finitepoisson}.
\begin{assumption}\label{ass:main}
    \hfill
    \begin{enumerate}
        \item The functions $\nu_J(\cdot)$, $b_i(\cdot)$ and $\lambda(\cdot)$ are infinitely differentiable almost everywhere in $\mathbb{R}^d$ for $i=1,\cdots,d$. And $\nu_J$ has a bounded support and $\nu_J(0)>0$.\label{subass:lambdanu}
        
        \item The jump-diffusion SDE  \eqref{sdeBL:finitepoisson} has a unique solution. The transition density $p_t(y,x)$ for the solution process  to move  from $y$ at time 0 to $x$ at time $t$, exists for all $x,y\in\mathbb{R}^d$ and all $t>0$.  In addition,  $p_t(y,x)$ is continuously differentiable with respect to $t$, twice continuously differentiable with respect to $x$ and $y$,  and  $p_t(y,x)>0$ for all $t>0$, $y, x\in\mathbb{R}^d$.\label{subass:existp}
        
        \item  For any $T>0$,  $\mathbb{P}\left( \sup_{0\le t\le T}|X_t|<\infty\right)=1$ .

    \end{enumerate}
\end{assumption}
There are some conditions for Assumption \ref{ass:main}-\eqref{subass:existp} to hold; for details, please refer to \cite[Theorem 2-29]{bichteler1987malliavin}.

\subsection{Preliminary: L\'{e}vy--Fokker--Planck equation}
The infinitesimal generator $\mathcal{L}$ associated with \eqref{sdeBL:finitepoisson} takes the form \cite{applebaum2009levy}:
{\begin{equation}\label{generator:poisson}
    \begin{aligned}
        \mathcal{L}f(x):=&\ \sum_{i=1}^d b_i(x)(\partial_if)(x) + \frac{\sigma^2}{2}\sum_{i,j=1}^d  (\partial_i\partial_jf)(x)  + \lambda(x)\int_{\mathbb{R}^d} (f(x+ z)-f(x))\nu_J(\d z),
    \end{aligned}
\end{equation}}for any   function $f: \mathbb{R}^d \to \mathbb{R}$ that possess bounded and continuous first and second derivatives. For every $t > 0$, $x \in \mathbb{R}^d$ and $z\in\mathbb{R}^d$, by Taylor's theorem 
{\small\begin{equation}\label{taylorexpansion:p}
    \begin{aligned}
         f(x+z)=&\ 
         f(x) + \sum_{i=1}^d z_i\int_0^1 (\partial_if)(x+\theta z)\d \theta,
    \end{aligned}
\end{equation}
}we can rewrite  \eqref{generator:poisson} as 
{\small \begin{equation*}
    \begin{aligned}
        \mathcal{L}f(x)=&\ \sum_{i=1}^d b_i(x)(\partial_if)(x) + \frac{\sigma^2}{2}\sum_{i,j=1}^d  (\partial_i\partial_jf)(x)  + \lambda(x)\int_{\mathbb{R}^d} \sum_{i=1}^d z_i\int_0^1 (\partial_if)(x+\theta z)\d \theta\ \nu_J(\d z).
    \end{aligned}
\end{equation*}}Then we obtain the expression of the adjoint operator of $\mathcal{L}$ using the integration by part as follows,
{\small \begin{equation*}
    \begin{aligned}
&\int_{\mathbb{R}^d}\mathcal{L}f (x)p_t(y,x)\d x 
         = \int_{\mathbb{R}^d}\left(\sum_{i=1}^d b_i(x)(\partial_if)(x) +  \frac{\sigma^2}{2}\sum_{i,j=1}^d (\partial_i\partial_jf)(x)\right)p_t(y,x)\d x \\
         &\ + \int_{\mathbb{R}^d}\int_{\mathbb{R}^d}\int_0^1 \sum_{i=1}^d z_i\lambda(x) (\partial_if)(x+\theta z)p_t(y,x)\d \theta\ \nu_J(\d z)\d x\\
         =&\ -\int_{\mathbb{R}^d}\left(\sum_{i=1}^d \partial_i \left( b_i(x) p_t(y,x) \right) - \frac{\sigma^2}{2}\sum_{i,j=1}^d \partial_i\partial_j p_t(y,x) \right)f(x)\d x \\
          &\ + \int_{\mathbb{R}^d}\int_{\mathbb{R}^d}\int_0^1 \sum_{i=1}^d z_i\lambda(x -\theta z) (\partial_if)(x)p_t(y,x -\theta z)\d \theta\ \nu_J(\d z)\d x\\
        =&\ -\int_{\mathbb{R}^d}\left(\sum_{i=1}^d \partial_i \left( b_i(x) p_t(y,x) \right) - \frac{\sigma^2}{2}\sum_{i,j=1}^d \partial_i\partial_j p_t(y,x) \right)f(x)\d x \\
        & - \int_{\mathbb{R}^d}\left(\int_{\mathbb{R}^d}\int_0^1 \sum_{i=1}^d z_i\partial_i \left(\lambda(x-\theta z) p_t(y,x-\theta z)\right)\d \theta\ \nu_J(\d z) \right)f(x) \d x  =:\ \int_{\mathbb{R}^d}f (x)\mathcal{L}^*p_t(y,x)\d x.
    \end{aligned}
\end{equation*}}

  The  L\'{e}vy--Fokker-Planck equation  describing the evolution of the density $p_t(y,x)$ of \eqref{sdeBL:finitepoisson} is the following partial differential equation $\partial_t p_t = \mathcal{L}^* p_t$ where $\mathcal{L}^*$ is for $x$ variable. Specifically, for any $x, y \in \mathbb{R}^d$ and $t > 0$, 
{\small \begin{equation}\label{FPE:poisson}
    \begin{aligned}
         \frac{\partial p_t(y,x)}{\partial t} 
        =&\ - \sum_{i=1}^d \partial_i \left( b_i(x) p_t(y,x) \right) +  \frac{\sigma^2}{2}\sum_{i,j=1}^d \partial_i\partial_j p_t(y,x) \\
        &\ -  \int_{\mathbb{R}^d}\int_0^1 \sum_{i=1}^d z_i\partial_i \left(\lambda(x-\theta z) p_t(y,x-\theta z)\right) \d \theta\ \nu_J(\d z).
    \end{aligned}
\end{equation}
}with the initial $p_0(y,x)=\delta_y(x)$. Here
 $\partial_i$ and $\partial_j$ represent $\frac{\partial}{\partial x_i}$ and $\frac{\partial}{\partial x_j}$, respectively.

 \begin{remark}
      Depending the order in   Taylor's series \eqref{taylorexpansion:p},  the form of the L\'{e}vy--Fokker-Planck equation \eqref{FPE:poisson} may have other representations \cite{sun2012fokker}.   Here, we select the zero-order expansion purely for simplicity, without the need to incorporate higher-order derivatives of the density $p_t$.
 \end{remark}

\subsection{Preliminary:  Probability flow of ODE and SDE}

Probability flow refers to the movement or evolution of probability distributions over time and is rooted in optimal transport. It can be understood as a path or curve in the space of probability distribution. 
Recently, there have been many numerical works that use the ODE as the flow map to realise probability flow for optimal transport, the Fokker-Planck equation, and generative artificial intelligence \cite{boffi2023probability,maoutsa2020interacting,song2020score, cao2024exploring, yoon2023score}. 

The idea of the probability flow is quite simple.  For the diffusion process, 
one rewrites the Fokker-Planck equation   as a continuity equation for a deterministic ODE flow map whose velocity field involves the probability distributions of the diffusion process.
This  establishes  the equivalence of probability flows  between the deterministic and stochastic dynamics. Specifically, denote the probability density of the solution for the It\^{o} diffusion SDE \eqref{SDE}    as $p_t(x)$ at time $t$, the corresponding Fokker--Planck equation can be written as
{\small\begin{equation*}
    \frac{\partial p_t(x)}{\partial t} 
        =\ - \sum_{i=1}^d \partial_i \left[\left( b_i(x)  -  \frac{\sigma^2}{2}\sum_{j=1}^d \  \partial_j( \log p_t(x) )  \right)p_t(x)\right]=:
        - \sum_{i=1}^d \partial_i  \left( \tilde{b}_{i}(x,t)  )p_t(x)\right),
\end{equation*}
}which has the form of a continuity equation and the vector field is $$\tilde{b}(x,t):=b(x) - \frac{\sigma^2}{2}\nabla\log p_t(x).$$ 
Then  at any time instant $t$, the solution process $X_t$ to \eqref{SDE} has the same (marginal) distribution with the solution process $\tilde{X}_t$ to the ODE system
$
    \frac{\d \tilde{X}_t}{\d t}=\tilde{b}(\tilde{X}_t),
$
if  $\tilde{X}_0$ and  $X_0$ have the same initial distribution. 
 The additional term $\nabla\log p_t(x)$ is known as the {\it score function}.  

Inspired by the probability flow in ODEs and It\^{o} diffusion SDEs, our strategy for the L\'{e}vy--Fokker--Planck equation \eqref{FPE:poisson} of the jump-diffusion process \eqref{sdeBL:finitepoisson} involves a few key steps. We convert the non-local integral term into part of the drift term. At the same time, we retain the Laplacian operator term. This transformation gives the L\'{e}vy--Fokker--Planck equation the form of a general Fokker--Planck equation for a diffusion process. Consequently, the probability flow equivalence between a jump-diffusion and a pure diffusion can be established.



\subsection{Probability flow equivalence between pure diffusion and jump-diffusion}
Formally, for the L\'{e}vy--Fokker--Planck equation \eqref{FPE:poisson}, we rewrite it as   
{\small \begin{equation*} 
    \begin{aligned}
        \frac{\partial p_t(y,x)}{\partial t} 
        =&\  - \sum_{i=1}^d \partial_i \left[\left( \underbrace{ b_i(x)  +  \int_{\mathbb{R}^d}\int_0^1  z_i \lambda(x-\theta z) \frac{ p_t(y,x-\theta z)}{p_t(y,x)} \d \theta\ \nu_J(\d z) }_{\mbox{``new'' drift}}  \right)p_t(y,x)\right]\\
        &\ + \frac{\sigma^2}{2}\sum_{i,j=1}^d \partial_i\partial_j p_t(y,x),
    \end{aligned}
\end{equation*}}so that it is in the form of  the Fokker-Planck equation for the diffusion SDE
\eqref{SDE} with the ``new'' drift term
and the same diffusion term $\sigma$.
 and the initial position at $y$. 
 However, to rigorously justify the meaning of this new diffusion ODE, one needs to deal with the singular ``new'' drift term at $t=0$, since $p_0(y,x)=\delta_{y}(x)$.

\begin{theorem}\label{theo:finite}
Assume that Assumption \ref{ass:main} holds and let $p_t(y,x)$ denote the solution of the L\'{e}vy--Fokker--Planck equation \eqref{FPE:poisson} with the initial  $p_0(y,x)=\delta_{y}(x)$. 
Let $\mathsf{N}_{x_0,\varepsilon} (\d x)$ denote the Gaussian distribution on $\mathbb{R}^d$, with the mean $x_0\in \mathbb{R}^d$ and the covariance matrix $\varepsilon I$, and $I$ is the $d\times d$ identity matrix.
Define the mollified probability distribution by the Gaussian kernel $\mathsf{N}_{x_0,\varepsilon}$
\begin{equation} \label{417}
p^{x_0,\varepsilon}_t(x):=\int_{\mathbb{R}^d}p_t(y,x)\, \mathsf{N}_{x_0,\varepsilon} (\d y).
\end{equation}
Then for any $\varepsilon>0$, and any $x_0\in \mathbb{R}^d$, the following It\^{o} stochastic differential equation is well-posed:
{\small \begin{equation}\label{sdeB:poissonvar}
       \left\{\begin{aligned}
            \d \widehat{X}^{x_0,\varepsilon}_t=&\ \left( b(\widehat{X}^{x_0,\varepsilon}_t)  +  \int_{\mathbb{R}^d}\int_0^1  z \lambda(\widehat{X}^{x_0,\varepsilon}_t-\theta z) \frac{ p^{x_0,\varepsilon}_t(\widehat{X}^{x_0,\varepsilon}_t-\theta z)}{p^{x_0,\varepsilon}_t(\widehat{X}^{x_0,\varepsilon}_t)} \d \theta\ \nu_J(\d z) \right)\d t  + \sigma \d B_t,\quad t>0,\\
            \widehat{X}^{x_0,\varepsilon}_0\sim &\ \mathsf{N}_{x_0,\varepsilon},
       \end{aligned}\right.
    \end{equation}}
and for any $A\in\mathcal{B}(\mathbb{R}^d)$ and $t\geq0$, 
    \begin{equation}
    \label{430}
        \mathbb{P}(X_t\in A| X_0=x_0)=\ \lim_{\varepsilon\to0}\mathbb{P}(\widehat{X}^{x_0,\varepsilon}_t\in A|\widehat{X}^{x_0,\varepsilon}_0=x_0).
    \end{equation}
\end{theorem}
\begin{proof}
Note that $p^{x_0,\varepsilon}_t $ is the distribution of the solution process $X^{x_0,\varepsilon}_t$ to the   jump-diffusion SDE \eqref{sdeBL:finitepoisson} with the non-singular  initial Gaussian distribution $\mathsf{N}_{x_0,\varepsilon}$:
\begin{equation}\label{sdeBL:finitepoissonvvar}
   \left\{\begin{aligned}
        \d X^{x_0,\varepsilon}_t=&\ b(X^{x_0,\varepsilon}_{t-})\d t + \sigma \d B_t +  J_t\d N_t,\quad t>0,\\
         X^{x_0,\varepsilon}_0\sim&\ \mathsf{N}_{x_0,\varepsilon}.
   \end{aligned}\right.
\end{equation}
  The   L\'{e}vy--Fokker--Planck equation of \eqref{sdeBL:finitepoissonvvar} reads
  \begin{equation}\label{FPE:var}
    \begin{aligned}
        \frac{\partial p^{x_0,\varepsilon}_t(x)}{\partial t} 
        =&\  - \sum_{i=1}^d \partial_i  \left(   \widehat{b}_i(x,t)   p^{x_0,\varepsilon}_t(x) \right) \ + \frac{\sigma^2}{2}\sum_{i,j=1}^d \partial_i\partial_j p^{x_0,\varepsilon}_t(x),
    \end{aligned}
\end{equation}
where the new drift term
\begin{equation} \label{eqn:bh}
     \widehat{b}_i(x,t) := b_i(x)  +  \int_{\mathbb{R}^d}\int_0^1  z_i \lambda(x-\theta z) \frac{ p^{x_0,\varepsilon}_t(x-\theta z)}{p^{x_0,\varepsilon}_t(x)} \d \theta\ \nu_J(\d z) 
\end{equation}
is well defined for all $t\ge 0$. 

Equation \eqref{FPE:var}   shares the same Fokker--Planck equation of the It\^{o} diffusion SDE \eqref{sdeB:poissonvar}.
Under Assumption \ref{ass:main}, it is known that, $|\widehat{b}^{x_0,\varepsilon}|\in L_{\mathrm{loc}}^{q_n}(\mathbb{R}_+;L^{p_n}(B_n(0)))$ for some $p_n,q_n$ satisfying $\frac{d}{p_n} + \frac{2}{q_n}<1$, where $B_n(0)$ is the closed ball in $\mathbb{R}^d$, centered at 0 with radius $n$. From this, it follows that there exists a unique strong solution to SDE \eqref{sdeB:poissonvar} for any initial point $y\in\mathbb{R}^d$, i.e.,
{\small  \begin{equation*}
    \begin{aligned}
        \widehat{X}^{x_0,\varepsilon}_t=&\ y  + \sigma B_t + \int_0^t \left( b(\widehat{X}^{x_0,\varepsilon}_t) +  \int_{\mathbb{R}^d}\int_0^1  z \lambda(\widehat{X}^{x_0,\varepsilon}_s-\theta z) \frac{ p^{x_0,\varepsilon}_s(\widehat{X}^{x_0,\varepsilon}_s-\theta z)}{p^{x_0,\varepsilon}_s(\widehat{X}^{x_0,\varepsilon}_s)} \d \theta\ \nu_J(\d z) \right)\d s,
    \end{aligned}
\end{equation*}}up to the explosion time, as established by \cite[Theorem 1.3]{zhang2011stochastic}. It is obvious that, $p^{x_0,\varepsilon}$ is the solution to the L\'{e}vy--Fokker--Planck equation \eqref{FPE:var}, and by the uniqueness of solution to \eqref{sdeB:poissonvar}, $p^{x_0,\varepsilon}$ is also the unique solution to   \eqref{FPE:var} on $\mathbb{R}^d\times[0,\infty)$. Thus the explosion time of $\widehat{X}^{x_0,\varepsilon}$ is infinite for  all $y\in\mathbb{R}^d$ almost surely under Assumption \ref{ass:main}. The first statement is proved.

The second statement is obvious by noting that, the normal distribution $ \mu_{x_0}^\varepsilon$ converges weakly to the Delta distribution $\delta_{x_0}(x)$. It is known  that $X^\varepsilon$ of \eqref{sdeBL:finitepoissonvvar} converges weakly to $X$ of 
\eqref{sdeBL:finitepoisson} on any bounded time intervals \cite[Theorem 7.1 \& 7.2]{billingsley2013convergence}. And from above argument we know that $\widehat{X}^{x_0,\varepsilon}$ shares all the marginal distributions with $X^\varepsilon$, thus
\begin{equation*}
    \lim_{\varepsilon\to0}\widehat{p}_t^\varepsilon(x_0,x)=p_t(x_0,x),\quad \forall t>0,\ x\in\mathbb{R}^d,
\end{equation*}
where $\widehat{p}^\varepsilon_t(x_0,x)$ denotes the transition density of \eqref{sdeB:poissonvar}.
\end{proof}
\begin{remark}

     It should be noted that our argument for probability flow equivalence means that the jump-diffusion process in \eqref{sdeBL:finitepoissonvvar} and the diffusion process in \eqref{sdeB:poissonvar}   share the same  {marginal} distributions at any given $t$ only if they also share the \emph{same} initial distributions. Therefore, if one wants to apply this probability flow equivalence to the transition kernel, from an arbitrary time $s$ to $t$, of the jump-diffusion process \eqref{sdeBL:finitepoisson}   by using  the diffusion process \eqref{sdeB:poissonvar}, it is necessary to construct the same distribution at $s$ for  the  diffusion process \eqref{sdeB:poissonvar} during this interval $[s,t]$.

\end{remark}

At the end of this subsection, we present  a short-time estimate of the transition kernel  $p_t(y,x)$, the solution of the L\'{e}vy--Fokker--Planck equation \eqref{FPE:poisson} with the initial $p_0(y,x)=\delta_{y}(x)$. This short-time analysis is crucial for deriving the Onsager--Machlup functional for jump-diffusion processes in a subsequent section. We present the following statement regarding this estimate.
\begin{lemma}\label{lem:tp}(\cite[Condition 1, 2 \& Theorem 1, 2]{yu2007closed})
    Under Assumption \ref{ass:main}, we have 
    \begin{equation}\label{expansion:p}
        p_t(y,x)=\ t^{-d/2}\exp\left(-\frac{C^{(-1)}(y,x)}{t}\right)\sum_{k=0}^\infty C^{(k)}(y,x)t^k + \sum_{k=1}^\infty D^{(k)}(y,x)t^k,
    \end{equation}
    for some functions $\{C^{(k)}\}_{k=-1}^\infty$ and $\{D^{(k)}\}_{k=1}^\infty$, and in particular,
    {\small \begin{equation*}
        \begin{aligned}
            C^{(-1)}(y,x)\geq&\ 0,\quad \forall x,y\in\mathbb{R}^d;\
            C^{(-1)}(y,x)=\ 0,\quad \mbox{if and only if}\quad x=y,\\
            C^{(0)}(x,x)=&\ \frac{1}{((2\pi)^{d}\det \Sigma(x))^{\frac{1}{2}}},\ 
            D^{(1)}(y,x)=\  \lambda(y)\nu_J(x-y).
        \end{aligned}
    \end{equation*}}
    And thus
    \begin{equation*}
        \begin{aligned}
            \lim_{t\downarrow0}\frac{p_t(y,x)}{t}=&\ \lambda(y)\nu_J(x-y),\quad x\neq y;\quad
            \lim_{t\downarrow0}\frac{p_t(x,x)}{t^{-\frac{d}{2}}}=&\ \frac{1}{((2\pi)^{d}\det \Sigma(x))^{\frac{1}{2}}}.
        \end{aligned}
    \end{equation*}
\end{lemma}

The first limit equation for $x\ne y$  can be intuitively understood as follows: when $t$ is small and
$x\neq y$, the dominant probability is attributed to a solitary jump from $y$ to $x$, governed by the  L\'{e}vy intensity  $\lambda(y)\nu_J(x-y)$ \cite{figueroa2014small,figueroa2018small,tran2023lamn}. And the second limit equation corresponds to the fact that $p_0(x,x)=\delta_{x}(x)=\infty$.

\begin{remark}
     By Lemma \ref{lem:tp}, we know that,
\begin{equation*}
    \begin{aligned}
        \lim_{t\downarrow0}\frac{p^{x_0,\varepsilon}_t(x)}{t}=\lim_{t\downarrow0}\frac{\int_{\mathbb{R}^d}p_t(y,x)\mathsf{N}_{x_0,\varepsilon}(y)\d y}{t}= \int_{\mathbb{R}^d}\lambda(y)\nu_J(x-y)\mathsf{N}_{x_0,\varepsilon}(y)\d y.
    \end{aligned}
\end{equation*}
This limit avoids the discussion of $x=y$ since it is a null set in terms of the Lebesgue measure. 
\end{remark}

In addition to the short-time estimate of the transition density, we also require a short-time estimate for its first-order derivative in a later section. For simplicity, we assume that:
\begin{assumption}\label{ass:uniform}
    Assume that, for every $t>0$ the sum of the first $n$ terms of the sequence \eqref{expansion:p} converges uniformly as $n \to \infty$.
\end{assumption}

\section{Derivation of the Onsager--Machlup Functional For Jump-Diffusions with Finite Jump Activity}\label{OMderivation}

In this section, we aim to estimate the infinitesimal tube-probability around a smooth path for the jump-diffusion \eqref{sdeBL:finitepoisson}.


\subsection{Preliminary: path space and measures induced by jump-diffusions}
In what follows, the jump-diffusions \eqref{sdeBL:finitepoisson} is discussed on a given time interval $[0,T]$. Recall that the solution process of a jump-diffusion model resides in the space of c\`{a}dl\`{a}g functions, which is denoted by
\begin{equation*}
    \mathbb{D}[0,T]:=\{f:[0,T]\to\mathbb{R}^d |\ f\ \mbox{is right continuous with left limit} \}.
\end{equation*}
Let $\mathbb{D}_{x_0}[0,T]$ denote the space of c\`{a}dl\`{a}g functions starting from point $x_0$, equipped with the uniform norm $\|\cdot\|$: $\|\psi\|=\sup_{t\in[0,T]}|\psi(t)|, \forall \psi\in \mathbb{D}_{x_0}[0,T].$
The $\sigma$-field generated by the uniform norm is equivalent to the projection $\sigma$-field formed by the finite-dimensional projections, as stated in \cite[Lemma 2.1]{chao2019onsager} and \cite[Page 87-90]{pollard2012convergence}. Here, a finite-dimensional projection $\pi_S$ is a mapping $\pi_S: \mathbb{D}_{x_0}[0,T] \to \mathbb{R}^{|S|}$ that sends $\psi$ to $\{\psi_{t_i} : t_i \in S \}$, where $S$ is a finite subset of $[0,T]$ and $|S|$ represents the number of elements in $S$. The solution process $X$ to \eqref{sdeBL:finitepoisson}  induces a measure on space $\mathbb{D}_{x_0}[0,T]$:
\begin{equation*}
    \mu_X(A)=\ \mathbb{P}(X\in A),\quad \forall A\in\mathcal{B}(\mathbb{D}_{x_0}[0,T]),
\end{equation*}
where $\mathcal{B}(\mathbb{D}_{x_0}[0,T])$ denotes the projection $\sigma$-field. For our purpose here, we focus on the closed tube sets in $\mathbb{D}_{x_0}[0,T]$, around a path $\psi$ with radius $\delta>0$, defined as
{\small \begin{equation*}
    \begin{aligned}
        K(\psi,\delta):=&\{\varphi\in \mathbb{D}_{x_0}[0,T]\mid\ \|\varphi-\psi\|\leq\delta \} =\left\{\varphi\in \mathbb{D}_{x_0}[0,T]\bigg|\ \sup_{t\in[0,T]\cap\mathbb{Q}}|\varphi(t)-\psi(t) |\leq\delta \right\},
    \end{aligned}
    \end{equation*}}where $\mathbb{Q}$ denotes all the rational points in $\mathbb{R}$. Following a similar approach to that described in \cite[Lemma 2.3]{huang2023most}, we can approximate any closed tube using cylinder sets as the following  lemma shows.
\begin{lemma}\label{cylinder}
For any $\psi\in\mathbb{D}_{x_0}[0,T]$ and $\delta>0$, there exist cylinder sets $\{I_n\}_{n=1}^\infty$ such that
{\small \begin{equation*}
    \begin{aligned}
        \mu_X(K(\psi,\delta))=&\ \lim_{n\to\infty}\mu_X(I_{n}(\psi,\delta)),\\
        I_{n}(\psi,\delta)=&\ \left\{\varphi\in\mathbb{D}_{x_0}[0,T] \bigg|\ \sup_{t=\{t_1,\cdots,t_n\}}|\varphi(t)-\psi(t)|\leq\delta \right\},
    \end{aligned}
\end{equation*}
}where $t_1,\cdots,t_n$ are some rational points in $[0,T]$. Moreover, when these points are arranged in ascending order, the maximum distance between any two adjacent points tends to zero as 
$n$ approaches infinity.
\end{lemma}
By this lemma, the behaviour of the jump-diffusion process $X$ with respect to the projection $\sigma$-field can be effectively approximated by its behavior on cylinder sets. The computation of the finite-dimensional distribution of a Markov process is facilitated by the Chapman--Kolmogorov equation, which functions analogously to a path integral. Lemma \ref{cylinder} serves as a crucial link between the path integral approach and the probability evaluation of a closed tube set. 

\subsection{Onsager--Machlup functional for jump-diffusion with finite jump activity}

We are now ready to demonstrate how to derive the Onsager--Machlup functionals for jump-diffusion processes with finite jump activity.

We first consider  the finite dimensional distribution of the solution process $X$ to the jump-diffusion SDE \eqref{sdeBL:finitepoisson}.
Let $0=t_0<t_1<\cdots<t_n=T$ be a time partition of time interval $[0,T]$, for simplicity we assume the partition is uniform, i.e., $\Delta t:=t_{i}-t_{i-1}=T/n$ for $i=1,\cdots,n$, and $\{A_i\subset\mathbb{R}^d\}_{i=1}^n$ are some Borel sets. 
To apply Theorem \ref{theo:finite}
on each interval $[t_{i-1}, t_i]$, we 
use the path integral  with  the classical Onsager--Machlup functional  (see  \eqref{eq:PISDE}) to represent the  transition density $\widehat{p}^{\varepsilon}_{t_i-t_{i-1}}(x_{i-1},x_i)$ of the It\^{o} diffusion SDE  with the Gaussian initial distribution \eqref{sdeB:poissonvar} as follows \cite{Kath1981path},
{\footnotesize\begin{equation}\label{subp:OM}
    \begin{split}
         \widehat{p}^{\varepsilon}_{t_i-t_{i-1}}(x_{i-1},x_i) 
        = \  \int_{C_{x_{i-1},x_i}[t_{i-1},t_i]} \exp &\bigg\{ - \int_{t_{i-1}}^{t_i} \bigg( \frac{1}{2\sigma^2}\bigg|\dot{\psi}^{(i)}_s   - \widehat{b}^{x_{i-1},\varepsilon}(\psi^{(i)}_s,s )
          \bigg|^2  \\
        &  + \frac{1}{2} \nabla\cdot  \widehat{b}^{x_{i-1},\varepsilon}(\psi^{(i)}_s,s ) \bigg) \d s \bigg\}~\mathcal{D}^{(i)}(\psi^{(i)}),
    \end{split}
\end{equation}}where   $\widehat{b}^{x_{i-1},\varepsilon}_i(x,t)$
is defined by \eqref{eqn:bh} for any $x_{i-1}$, $x$ and $t$,
$p^{x_{i-1},\epsilon}_s$  is defined via \eqref{417}, 
 $C_{x_{i-1},x_i}[t_{i-1},t_i]$ denotes the space of continuous paths starting from $x_{i-1}$ at time $t_{i-1}$ to $x_i$ at time $t_i$.
 $\psi^{(i)}\in C_{x_{i-1},x_i}[t_{i-1},t_i]$ can be regarded as the restriction of a path $\psi\in C_{x_0,x_T}[0,T]$ on the interval $[t_{i-1}, t_i]$.
 The symbol  $\mathcal{D}^{(i)}$ is the formal path measure on  $C_{x_{i-1},x_i}[t_{i-1},t_i]$. By  Theorem \ref{theo:finite} and    \eqref{subp:OM},  we represent the finite-dimensional distribution of $X$ as follows,
{ \footnotesize\begin{equation}\label{finitedistribution}
    \begin{aligned}
        &\ \mathbb{P}(X_{t_i}\in A_i,\ i=1,\cdots,n)=\ \int_{A_1\times\cdots \times A_n}\prod_{i=1}^n p_{t_i-t_{i-1}}(x_{i-1},x_i)\d x_1\cdots\d x_n\\
        \overset{*}{=}&\ \int_{A_1\times\cdots \times A_n}\prod_{i=1}^n \lim_{\varepsilon\to0}~\widehat{p}^{\varepsilon}_{t_i-t_{i-1}}(x_{i-1},x_i)\d x_1\cdots\d x_n\\
        \overset{**}{=}&\  \int_{\prod_{i=1}^{n}A_i}\prod_{i=1}^n \lim_{\varepsilon\to0}~\int_{C_{x_{i-1},x_i}[t_{i-1},t_i]} \exp\left\{ - \int_{t_{i-1}}^{t_i} \left( \frac{1}{2\sigma^2}\left|\dot{\psi}^{(i)}_s -  \widehat{b}^{x_{i-1},\varepsilon}(\psi^{(i)}_s,s)    \right|^2 \right.\right.\\
        &\ \left.\left. + \frac{1}{2} \nabla\cdot \widehat{b}^{x_{i-1},\varepsilon}(\psi^{(i)}_s,s )  \right)\d s\right\} \mathcal{D}^{(i)}(\psi^{(i)})  \d x_i\\
        =&\  \int_{\prod_{i=1}^nA_i} \lim_{\varepsilon\to0} \int_{\prod_{i=1}^{n}C_{x_{i-1},x_i}[t_{i-1},t_i]} \exp\left\{ - \sum_{i=1}^n\int_{t_{i-1}}^{t_i} \left( \frac{1}{2\sigma^2}\left|\dot{\psi}^{(i)}_s -  \widehat{b}^{x_{i-1},\varepsilon}(\psi^{(i)}_s,s )   \right|^2 \right.\right.\\
        &\ \left.\left. + \frac{1}{2} \nabla\cdot \widehat{b}^{x_{i-1},\varepsilon}(\psi^{(i)}_s,s )  \right)\d s\right\} \left(\prod_{i=1}^n\mathcal{D}^{(i)}(\psi^{(i)})\right) \d x_1\cdots\d x_{n}.
    \end{aligned}
\end{equation}}We used  the result of Theorem \ref{theo:finite} for equality $\overset{*}{=}$, which is to replace each transition kernel $p_{t_i-t_{i-1}}(x_{i-1},x_i)$ by the transition kernel $\widehat{p}^{\varepsilon}_{t_i-t_{i-1}}(x_{i-1},x_i)$ of the corresponding diffusion process \eqref{sdeB:poissonvar} in the limit $\varepsilon\to0$. The equality $\overset{**}{=}$ is the application of \eqref{subp:OM}. 

The term $\widehat{b}^{x_{i-1},\varepsilon}(\psi^{(i)}_s,s )$  in \eqref{finitedistribution} contains the ratio  $ \frac{p^{\psi_{t_{i-1}},\varepsilon}_{s}(\psi^{(i)}_s-\theta z)}{p^{\psi_{t_{i-1}},\varepsilon}_{s}(\psi^{(i)}_s)}$ by its definition \eqref{eqn:bh}. We next focus on the short time limit $(\Delta t=T/n\to0)$ of this ratio.
We know that for a path $\psi \in C_{x_0}[0,T]$, divided into partition segments $\{\psi^{(i)}_{s} := \psi_{s+t_{i-1}}, s \in [0, \Delta t]\}_{i=1}^n$, for the $i$-th interval, by Lemma \ref{lem:tp} we obtain the following asymptotic result for small $s$,
{\small \begin{equation*}
    \begin{aligned}
        \frac{p^{\psi_{t_{i-1}},\varepsilon}_{s}(\psi^{(i)}_s-\theta z)}{p^{\psi_{t_{i-1}},\varepsilon}_{s}(\psi^{(i)}_s)}
        =&\ \frac{ \int_{\mathbb{R}^d} \left(D^{(1)}(y,\psi^{(i)}_s-\theta z)s +  \mathcal{O}(s^2)\right)\mathsf{N}_{\psi_{t_{i-1}},\varepsilon}(\d y)}{ \int_{\mathbb{R}^d} \left(D^{(1)}(y,\psi^{(i)}_s )s + \mathcal{O}(s^2)\right)\mathsf{N}_{\psi_{t_{i-1}},\varepsilon}(\d y) }\\
        \xrightarrow{\varepsilon\to0}&\  \frac{   D^{(1)}(\psi_{t_{i-1}},\psi^{(i)}_s-\theta z)   }{ D^{(1)}(\psi_{t_{i-1}},\psi^{(i)}_s ) }  + \mathcal{O}(T/n)\\
        =&\ \frac{  \nu_J( -\theta z )   }{ \nu_J (0) } + \mathcal{O}(\|\psi^{(i)}-\psi_{t_{i-1}}\|_{[t_{i-1},t_i]}) + \mathcal{O}(T/n).
    \end{aligned}
\end{equation*}}The term $\mathcal{O}(\|\psi^{(i)} - \psi_{t_{i-1}}\|_{[t{i-1},t_i]})$ arises from the expansion of $\nu_J(\psi^{(i)}_s - \psi_{t_{i-1}} - \theta z)$ and $\nu_J(\psi^{(i)}_s - \psi_{t_{i-1}})$ around $\nu_J(-\theta z)$ and $\nu_J(0)$, respectively. Here, $\|\cdot\|_{[t_{i-1},t_i]}$ denotes the supremum norm on $C[t_{i-1},t_i]$. The remaining term $\mathcal{O}(T/n)$ depends on $\psi$ as well as the quantities $z$ and $\theta$. Since $\nu_J$ has bounded support and $0\leq \theta\leq1$, we know that for any family of uniformly bounded paths, the corresponding remaining terms $\mathcal{O}(T/n)$ uniformly tend to 0 as $\Delta t \to 0$. Under Assumption \ref{ass:uniform}, the asymptotic estimate for the corresponding gradient (with respect to the variable $\psi^{(i)}_s$) is
{\small   \begin{equation*}
    \begin{aligned}
        &\nabla\frac{p^{\psi_{t_{i-1}},\varepsilon}_{s}(\psi^{(i)}_s-\theta z)}{p^{\psi_{t_{i-1}},\varepsilon}_{s}(\psi^{(i)}_s)}=\ \nabla\frac{ \int_{\mathbb{R}^d} \left(D^{(1)}(y,\psi^{(i)}_s-\theta z)s +  \mathcal{O}(s^2)\right)\mathsf{N}_{\psi_{t_{i-1}},\varepsilon}(\d y)}{ \int_{\mathbb{R}^d} \left(D^{(1)}(y,\psi^{(i)}_s )s + \mathcal{O}(s^2)\right)\mathsf{N}_{\psi_{t_{i-1}},\varepsilon}(\d y) }\\
        \xrightarrow{\varepsilon\to0}&\  \nabla\frac{   D^{(1)}(\psi_{t_{i-1}},\psi^{(i)}_s-\theta z)   }{ D^{(1)}(\psi_{t_{i-1}},\psi^{(i)}_s ) }  + \mathcal{O}(T/n)\\
        =&\ \frac{  \nabla\nu_J( -\theta z)\nu_J(0) - \nu_J(-\theta z)\nabla\nu_J(0)  }{ (\nu_J (0))^2 } + \mathcal{O}(\|\psi^{(i)}-\psi_{t_{i-1}}\|_{[t_{i-1},t_i]}) + \mathcal{O}(T/n).
    \end{aligned}
\end{equation*}}
Thus the   exponential term of path integral representation  in \eqref{finitedistribution} becomes 
{\small   \begin{equation}\label{PI:poisson}
    \begin{aligned}
        &\ \sum_{i=1}^n\int_{t_{i-1}}^{t_i} \left( \frac{1}{2\sigma^2}\bigg|\dot{\psi}^{(i)}_s   - b(\psi^{(i)}_s)
        - \int_{\mathbb{R}^d}\int_0^1z\lambda(\psi^{(i)}_s -\theta z)\frac{p^{\psi_{t_{i-1}},\varepsilon}_{s}(\psi^{(i)}_s-\theta z)}{p^{\psi_{t_{i-1}},\varepsilon}_{s}(\psi^{(i)}_s)}\d\theta\ \nu_J(\d z)    \bigg|^2  \right.\\
        &\ \left.  + \frac{1}{2} \nabla\cdot \left[ b(\psi^{(i)}_s) + \int_{\mathbb{R}^d}\int_0^1 z\lambda(\psi^{(i)}_s-\theta z) \frac{p^{\psi_{t_{i-1}},\varepsilon}_{s}(\psi^{(i)}_s-\theta z)}{p^{\psi_{t_{i-1}},\varepsilon}_{s}(\psi^{(i)}_s)}\d\theta\ \nu_J(\d z) \right] \right)\d s \\ &\xrightarrow{\varepsilon\to0}\ \sum_{i=1}^n\int_{t_{i-1}}^{t_i} \left( \frac{1}{2\sigma^2}\bigg|\dot{\psi}^{(i)}_s   - b(\psi^{(i)}_s) - \int_{\mathbb{R}^d}\int_0^1z\lambda(\psi^{(i)}_s -\theta z) \right.\\
        &\ \left. \times \frac{\nu_J(\psi^{(i)}_s-\psi_{t_{i-1}}-\theta z)  }{ \nu_J(\psi^{(i)}_s-\psi_{t_{i-1}})}\d\theta\ \nu_J(\d z)    \bigg|^2  + \frac{1}{2} \nabla\cdot \left[ b(\psi^{(i)}_s) \right.\right.\\
        &\ \left.\left.  + \int_{\mathbb{R}^d}\int_0^1 z\lambda(\psi^{(i)}_s-\theta z) \frac{\nu_J(\psi^{(i)}_s-\psi_{t_{i-1}}-\theta z)  }{ \nu_J(\psi^{(i)}_s-\psi_{t_{i-1}}) }\d\theta\ \nu_J(\d z) \right]   + \mathcal{O}(T/n)  \right)\d s\\
        &= \sum_{i=1}^n\int_{t_{i-1}}^{t_i} \left( \frac{1}{2\sigma^2}\bigg|\dot{\psi}^{(i)}_s   - b(\psi^{(i)}_s) - \int_{\mathbb{R}^d}\int_0^1z\lambda(\psi^{(i)}_s -\theta z) \frac{\nu_J(-\theta z)  }{ \nu_J(0)}\d\theta\ \nu_J(\d z)    \bigg|^2 \right. \\
        &\ + \frac{1}{2} \nabla\cdot \left[ b(\psi^{(i)}_s)   + \ell(\psi_s^{(i)}) \right]  + \frac{1}{2}\tilde{\ell}(\psi^{(i)}_s) + \mathcal{O}(\|\psi^{(i)}-\psi_{t_{i-1}}\|_{[t_{i-1},t_i]}) + \mathcal{O}(T/n)  \bigg)\d s,
    \end{aligned}
\end{equation}}where $\ell_J $ and $\tilde{\ell}_J$ are given by \eqref{lJ}. By substituting \eqref{PI:poisson} into \eqref{finitedistribution}, we obtain
    {\small \begin{equation}\label{PI:finite}
    \begin{aligned}
      \mathbb{P}(X_{t_i}\in A_i,\ i=1,\cdots,n) 
        =&\ \int_{\substack{\psi\in C_{x_0,x_T}[0,T] \\ \psi_{t_i}\in A_i,i=1,\cdots,n}}\exp\left\{ - \frac{1}{2}\int_0^T\left( \frac{1}{\sigma^2}\left|\dot{\psi}_s - b(\psi_s) - \ell(\psi_s) \right|^2 \right.\right. \\
        &\ + \nabla\cdot (b(\psi_s) + \ell(\psi_s)) + \tilde{\ell}(\psi_s) \bigg)\d s \\
        &\ + \sum_{i=1}^n(T/n)\mathcal{O}\left(\|\psi^{(i)}-\psi_{t_{i-1}}\|_{[t_{i-1},t_i]}\right) + \mathcal{O} (T/n) \bigg\}\mathcal{D}(\psi).
    \end{aligned}
\end{equation}}For the time partition $0=t_0<t_1<\cdots<t_n=T$, let $A_i=B_{\delta}(\phi_{t_i})$ for some  $\phi\in C^2_{x_0}[0,T]$, where $C^2_{x_0}[0,T]$ denotes the space of the continuous paths starting form $x_0$ with continuous first-order derivatives, and $\delta>0$ is a constant. Under Lemma \ref{cylinder}, letting $n$ to tend to infinity and for small enough $\delta$, all the remaining $\mathcal{O}$ terms for the trajectories in the tube $K(\phi,\delta)$ are uniform bounded as discussed before, thus we have the following results,
  {\small \begin{equation*} 
    \begin{aligned}
     \mu_X(K(\phi,\delta))
        =&\  \lim_{n\to\infty}\mathbb{P}(X_{t_i}\in B_{\delta}(\phi_{t_i}) ,\ i=1,\cdots,n)\\
       \sim&\ \int_{\psi\in C_{x_0,x_T}[0,T],\|\psi-\phi\|\leq\delta}\exp\left\{ - \int_0^T \left( \frac{1}{2\sigma^2}\bigg|\dot{\psi}_s   - b(\psi_s)
       - \ell_J(\psi_s)  \bigg|^2 \right.\right.\\
       &\  \left.\left. + \frac{1}{2} \nabla\cdot \left[ b(\psi_s)
       + \ell_J(\psi_s)  \right]  + \frac{1}{2}\tilde{\ell}_J(\psi_s)\right)\d s \right\}\mathcal{D}(\psi),\quad \delta\downarrow0.
    \end{aligned}
\end{equation*}}

This  path integral representation is typically seen as a formal representation. We can also 
directly compute $\mu_X(K(\phi,\delta))$ 
by starting with   the finite dimensional distribution, by using the Girsanov theorem. The measure $\mu_{\widehat{X}^{x_{i-1},\varepsilon}}$, induced by the diffusion process $\widehat{X}^{x_{i-1},\varepsilon}$ in the SDE \eqref{sdeB:poissonvar}, is absolutely continuous with respect to the Wiener measure $\mu_B$. By using the Girsanov theorem for the process \eqref{sdeB:poissonvar} we obtain,
{ \small  \begin{equation*}
    \begin{aligned}
       &p_{t_i-t_{i-1}}(x_{i-1},x_i)=\lim_{\varepsilon\to0}\widehat{p}^\varepsilon_{t_i-t_{i-1}}(x_{i-1},x_i)=\  \lim_{\varepsilon\to0} \int_{C_{x_{i-1},x_i}[t_{i-1},t_i]} \frac{\d\mu_{\widehat{X}^{x_{i-1},\varepsilon}}}{\d\mu_B}(\psi)\mu_B(\d\psi)\\
       =&\ \lim_{\varepsilon\to0}\int_{C_{x_{i-1},x_i}[t_{i-1},t_i]} \exp\left\{ -\int_{t_{i-1}}^{t_i} \frac{1}{2\sigma^2}|\widehat{b}^{x_{i-1},\varepsilon}(\psi_s)|^2 + \int_{t_{i-1}}^{t_i} \widehat{b}^{x_{i-1},\varepsilon}(\psi_s)\d\psi_s \right\}\mu_B(\d \psi)\\
       \overset{*}{=}&\ \lim_{\varepsilon\to0}\int_{C_{x_{i-1},x_i}[t_{i-1},t_i]} \exp\left\{ - \int_{t_{i-1}}^{t_i} \frac{1}{2\sigma^2}\bigg| b(\psi_s)
       +  \int_{\mathbb{R}^d}\int_0^1  z \lambda(\psi_s-\theta z) \right. \\
       &\ \times\frac{ p^{x_{i-1},\varepsilon}_t(\psi_s-\theta z)}{p^{x_{i-1},\varepsilon}_t(\psi_s)} \d \theta\ \nu_J(\d z)   \bigg|^2\d s  + \int_{t_{i-1}}^{t_i}\left(b(\psi_s)
       +  \int_{\mathbb{R}^d}\int_0^1  z \lambda(\psi_s-\theta z) \right. \\
       &\ \left. \times \frac{ p^{x_{i-1},\varepsilon}_t(\psi_s-\theta z)}{p^{x_{i-1},\varepsilon}_t(\psi_s)} \d \theta\ \nu_J(\d z) \right)\circ \d \psi_s  - \frac{1}{2}\nabla\cdot\int_{t_{i-1}}^{t_i}\left(b(\psi_s)
       +  \int_{\mathbb{R}^d}\int_0^1  z \lambda(\psi_s-\theta z)\right. \\
       &\ \left.\left. \times \frac{ p^{x_{i-1},\varepsilon}_t(\psi_s-\theta z)}{p^{x_{i-1},\varepsilon}_t(\psi_s)} \d \theta\ \nu_J(\d z) \right)\d s
       \right\}\mu_B(\d\psi)\\
       \overset{**}{=}&\ \int_{C_{x_{i-1},x_i}[t_{i-1},t_i]} \exp\left\{ - \int_{t_{i-1}}^{t_i}  \frac{1}{2\sigma^2}\bigg| b(\psi_s)
       +  \int_{\mathbb{R}^d}\int_0^1  z \lambda(\psi_s-\theta z) \right.\\
       &\ \times \frac{ \nu_J(\psi_s-x_{i-1}-\theta z)}{\nu_J(\psi_s-x_{i-1})} \d \theta\ \nu_J(\d z)   \bigg|^2\d s  + \int_{t_{i-1}}^{t_i}\left(b(\psi_s)
       +  \int_{\mathbb{R}^d}\int_0^1  z \lambda(\psi_s-\theta z) \right. \\
       &\ \left. \times \frac{ \nu_J(\psi_s-x_{i-1}-\theta z)}{\nu_J(\psi_s-x_{i-1})} \d \theta\ \nu_J(\d z) \right)\circ\d \psi_s - \frac{1}{2}\nabla\cdot\int_{t_{i-1}}^{t_i}\bigg(b(\psi_s)
       \\
       &\ \left.\left.  +  \int_{\mathbb{R}^d}\int_0^1  z \lambda(\psi_s-\theta z) \frac{ \nu_J(\psi_s-x_{i-1}-\theta z)}{\nu_J(\psi_s-x_{i-1})} \d \theta\ \nu_J(\d z) \right) \d s + \mathcal{O}(T/n) \right \}\mu_B(\d\psi)\\
        \overset{***}{=}&\ \int_{C_{x_{i-1},x_i}[t_{i-1},t_i]} \exp\left\{ - \int_{t_{i-1}}^{t_i} \frac{1}{2\sigma^2}\left| b(\psi_s)
       + \ell(\psi_s) \right|^2\d s + \int_{t_{i-1}}^{t_i}\left(b(\psi_s)\right. \right. \\
       &\ \left. 
       + \ell(\psi_s) \right)\d \psi_s  - \frac{1}{2}\int_{t_{i-1}}^{t_i}\tilde{\ell }(\psi_s)\d s  +  \mathcal{O}(\|\psi - \psi_{t_{i-1}}\|_{[t_{i-1},t_i]}) + \mathcal{O}(T/n) \bigg\}\mu_B(\d\psi),
    \end{aligned}
\end{equation*}}where $\psi$ is understood as a sample trajectory of Brownian motion and the integrals with differential elements $\d \psi_s$ are understood in the sense of It\^{o} stochastic integrals. And $\circ\d \psi_s$ in equality $\overset{*}{=}$ denotes the Stratonovich integral. We use Stratonovich integral since we need to ``separate'' the dominant term from the integrands of equalities $\overset{**}{=}$ and $\overset{***}{=}$ (that is we ``separate'' $\nu_J(\psi_s-x_{i-1})$ as $\nu_J(0) +\mathcal{O}(\psi_s-x_{i-1})$ and similar separations for $\nu_J(\psi_s-x_{i-1}-\theta z)$, $\nabla \nu_J(\psi_s -x_{i-1}-\theta z)$ and $\nabla\nu_J(\psi_s-x_{i-1})$), and such calculations for Stratonovich integral coincide with the classical calculus \cite{stratonovich1966new}.

Consequently, we obtain the following equation:
    { \begin{equation}\label{PI:finite2}
    \begin{aligned}
     &\mathbb{P}(X_{t_i}\in A_i,\ i=1,\cdots,n)
        =\ \int_{\substack{\psi\in C_{x_0,x_T[0,T]} \\ \psi_{t_i}\in A_i,i=1,\cdots,n}}\exp\left\{ - \int_0^T \left( \frac{1}{2\sigma^2}\left| b(\psi_s) + \ell(\psi_s)  \right|^2\d s  \right.\right.
       \\
       &\  + \int_0^T\left(b(\psi_s) +  \ell(\psi_s) \right)\d \psi_s -\frac{1}{2}\int_0^T\tilde{\ell}(\psi_s)\d s\\ 
       &\  \left.  + \sum_{i=1}^n(T/n)\mathcal{O}\left( \|\psi - \psi_{t_{i-1}}\|_{[t_{t-1},t_i]} \right) + \mathcal{O}(T/n) \right\}\mu_B(\d\psi).
    \end{aligned}
\end{equation}}

Again, by Lemma \ref{cylinder}, letting $n$ to tend to infinity and for small enough $\delta$, all the remaining $\mathcal{O}$ terms for the trajectories in the tube $K(\phi,\delta)$ are uniform bounded as discussed before, we obtain that,
 { \begin{equation*} 
    \begin{aligned}
     \mu_X(K(\phi,\delta))
        =&\  \lim_{n\to\infty}\mathbb{P}(X_{t_i}\in B_{\delta}(\phi_{t_i}) ,\ i=1,\cdots,n)\\
       =&\ \int_{\psi\in C_{x_0,x_T}[0,T],\|\psi-\phi\|\leq\delta}\exp\left\{ - \int_0^T \left( \frac{1}{2\sigma^2}\left| b(\psi_s)
       + \ell_J(\psi_s)  \right|^2\d s \right.\right.\\
        &\ \left.  + \int_0^T\left(b(\psi_s)
       + \ell_J(\psi_s)\right)\d \psi_s - \frac{1}{2}\int_0^T \tilde{\ell}_J(\psi_s)\d s  \right\}\mu_B(\d\psi)\\
        \sim &\ \exp\left\{ - \frac{1}{2}\int_0^T \left( \frac{1}{\sigma^2}\left|\dot{\phi}_s   - b(\phi_s)
       - \ell_J(\phi_s)  \right|^2 \right.\right.\\
       &\  +  \nabla\cdot \left[ b(\phi_s)
       + \ell_J(\phi_s) \right] + \tilde{\ell}_J(\phi_s) \bigg)\d s \bigg\}\mu_B(K(x_0,\delta)),\quad \delta\downarrow0,
    \end{aligned}
\end{equation*}}
where the last asymptotic calculation follows the classical Onsager--Machlup method, as detailed in \cite{Durr1978,Ikeda1980}, to derive the Onsager--Machlup functional.

To summarize, we have the following statement.

\begin{theorem}\label{theo:PI}
   For the jump-diffusion SDE \eqref{sdeBL:finitepoisson}, assuming that Assumption \ref{ass:main} and \ref{ass:uniform} are satisfied, for any $\phi \in C^2_{x_0}[0,T]$ and sufficiently small $\delta > 0$, it follows that
 { \begin{equation}\label{muX}
    \begin{aligned}
     \mu_X(K(\phi,\delta))
       \sim &\ \int_{\psi\in C_{x_0,x_T}[0,T],\|\psi-\phi\|\leq\delta}\exp\left\{ - \int_0^T \left( \frac{1}{2\sigma^2}\left|\dot{\psi}_s   - b(\psi_s)
       - \ell_J(\psi_s)  \right|^2 \right.\right.\\
       &\  \left.\left. + \frac{1}{2} \nabla\cdot \left[ b(\psi_s)
       + \ell_J(\psi_s)  \right]  + \frac{1}{2}\tilde{\ell}_J(\psi_s)\right)\d s \right\}\mathcal{D}(\psi)\\
       \sim&\ \exp\left\{ - \frac{1}{2}\int_0^T \left( \frac{1}{\sigma^2}\left|\dot{\phi}_s   - b(\phi_s)
       - \ell_J(\phi_s)  \right|^2 \right.\right.\\
       &\     +  \nabla\cdot \left[ b(\phi_s)
       + \ell_J(\phi_s)  \right]   + \tilde{\ell}_J(\psi_s) \bigg)\d s \bigg\}\mu_B(K(x_0,\delta)) ,\quad \delta \downarrow0,
    \end{aligned}
\end{equation}}
where $\ell_J $ and $\tilde{\ell}_J$ are given by \eqref{lJ} and $\mu_B$ is the Wiener measure. 
\end{theorem}

\begin{remark}
Deriving the results of Theorem \ref{theo:PI} relies on the  probability flow equivalence between jump-diffusion and pure diffusion processes (Theorem \ref{theo:finite}), and the existing result of the OM functional for the diffusion process. If the diffusion coefficient $\sigma=0$, the L\'{e}vy--Fokker--Planck equation \eqref{FPE:var} in the form of continuity equation 
{\small \begin{equation*} 
    \begin{aligned}
        \frac{\partial p^{x_0,\varepsilon}_t(x)}{\partial t} 
        =&\  - \sum_{i=1}^d \partial_i  \left[  \underbrace{\left(b_i(x)  +  \int_{\mathbb{R}^d}\int_0^1  z_i \lambda(x-\theta z) \frac{ p^{x_0,\varepsilon}_t(x-\theta z)}{p^{x_0,\varepsilon}_t(x)} \d \theta\ \nu_J(\d z) \right)}_{\mbox{``new'' drift} } p^{x_0,\varepsilon}_t(x) \right],
    \end{aligned}
\end{equation*}} corresponds to a deterministic system where the velocity field is defined by the ``new'' drift above. However, it is important to note that deterministic systems do not possess Onsager--Machlup functionals. Therefore, our probability flow approach is only applicable to jump-diffusion processes and cannot be straightforwardly extended to pure jump processes.
\end{remark}
\begin{remark}
For simplicity and clarity, this paper mainly concentrates on cases involving additive Brownian noise (constant $\sigma$).  It is both possible and simple to extend our findings to situations with multiplicative Brownian noise.  The probability flow equivalence between jump-diffusion and diffusion processes, when involving multiplicative Brownian noise, follows a similar approach to that described in \eqref{FPE:var}. Additionally, the literature provides established methods for the short-time approximation of density in jump-diffusion processes with multiplicative noise \cite{yu2007closed}, as well as techniques for the Onsager--Machlup functional in diffusion processes with such noise \cite{zeitouni1988existence}. Consequently, the conclusions of Theorem \ref{theo:PI} can be  applied to jump-diffusion processes with multiplicative Brownian noise, provided the diffusion coefficient $\sigma(x)$ adheres to specific established conditions detailed in references \cite{yu2007closed, zeitouni1988existence}. The OM functional in \eqref{muX} will be updated as follows for multiplicative noise:
\begin{equation*}
    \begin{aligned}
        &\ \frac{1}{2}\int_0^T \left[ \left(\dot{\phi}_s   - b(\phi_s)
       - \ell_J(\phi_s)  \right)^\top\Sigma(\phi_s)\left(\dot{\phi}_s   - b(\phi_s)
       - \ell_J(\phi_s)  \right)   + \tilde{\ell}_J(\psi_s) \right. \\
            +&\  \frac{1}{\sqrt{|\Sigma(\phi_s)|}}\sum_{i=1}^d\frac{\partial}{\partial x_i}\left[\left( b_i(\phi_s)
       + (\ell_J)_i(\phi_s) \right)\sqrt{|\Sigma(\phi_s)|} \right]  + \frac{1}{\sqrt{|\Sigma(\phi_s)|}}\sum_{i=1}^d(\ell_J)_i\frac{\partial}{\partial x_i}\sqrt{|\Sigma(\phi_s)|}\bigg]\d s
    \end{aligned}
\end{equation*}
where $\Sigma(\cdot)=(\sigma(\cdot)\sigma(\cdot)^\top)^{-1}$, and $|\Sigma|$
is the determinate of the matrix $\Sigma$.
\end{remark}

\section{Jump-Diffusion with Infinite Jump Activity}\label{sec:infinite}

In the previous section, we  have derived the Onsager--Machlup functional for jump-diffusion processes with finite jump activity. We next study the more challenging case of  jump-diffusion processes with infinite activity.

Consider the following jump-diffusion SDE with infinite jump activity:
\begin{equation}\label{sdeBL}
   \left\{\begin{aligned}
         \d X^i_t=&\ b_i(X_{t-})\d t +  \sigma \d B^j_t + \int_{|z|<1} F_i(X_{t-},z)  \tilde{\mathcal{N}}(\d t,\d z) \\
        &\ + \int_{|z|\geq1}G_i(X_{t-},z) \mathcal{N}(\d t,\d z),\quad t>0,\quad i=1,\cdots,d,\\
          X_0=&\ x_0\in\mathbb{R}^d,
   \end{aligned}\right.
\end{equation}
where $X_t\in\mathbb{R}^d$ denotes a random particle's position, the mappings $b_i:\mathbb{R}^d\to\mathbb{R}$, $F_i:\mathbb{R}^d\times\mathbb{R}^d\to\mathbb{R}$ and $G_i:\mathbb{R}^d\times\mathbb{R}^d\to\mathbb{R}$ are all assumed to be measurable for $1\leq i\leq d$, $\sigma$ is a positive constant, and $B=(B^1,\cdots,B^d)$ is a standard Brownian motion in $\mathbb{R}^d$ and $\mathcal{N}$ an independent Poisson random measure on  $\mathbb{R}_+\times(\mathbb{R}^d\backslash\{0\})$ with associated compensator $\tilde{\mathcal{N}}$ and intensity measure $\nu$, where $\nu$ is a L\'{e}vy measure which is assumed to be absolutely continuous respect to the Lebesgue measure, i.e., $\tilde{\mathcal{N}}(\d t,\d z)=\mathcal{N}(\d t,\d z)-\d t\nu(\d z)$ and satisfies $\int_{\mathbb{R}^d}(1\wedge |z|^2)\nu(\d z)<\infty.$

For jump-diffusion SDE \eqref{sdeBL}, we make the following assumption.
\begin{assumption}\label{assum:bnu}
The stochastic differential equation \eqref{sdeBL} has a unique solution. The transition density $p_t(y,x)$ exists for all $x,y\in\mathbb{R}^d$ and all $t>0$. The density $p_t(y,x)$ is continuously differentiable with respect to $t$, at least twice continuously differentiable with respect to $x$ and $y$. 
\end{assumption}

There are some sufficient conditions for Assumption \ref{assum:bnu} to hold, see for example \cite{cass2009smooth}.

\subsection{L\'{e}vy--Fokker--Planck equation for jump-diffusion with infinite jump activity and Probability flow equivalence}

The generator of the solution process to \eqref{sdeBL} is defined on $C_0^2(\mathbb{R}^d)$ as follows \cite{applebaum2009levy,duan2015introduction},
{ \begin{equation}\label{generator}
    \begin{aligned}
         (\mathcal{L}f)(x)=&\ \sum_{i=1}^d b_i(x)(\partial_if)(x) + \frac{\sigma^2}{2}\sum_{i,j=1}^d  (\partial_i\partial_jf)(x) + \int_{|z|\geq 1}\left[ f(x+G(x,z)) - f(x) \right]\nu(\d z)\\
        &\ + \int_{|z|<1}\left[ f(x+F(x,z)) - f(x) - \sum_{i=1}^d F_i(x,z)(\partial_i f)(x)\right]\nu(\d z),
    \end{aligned}
\end{equation}}for each function $f \in C^2_0(\mathbb{R}^d)$ and each point $x \in \mathbb{R}^d$. Again, Taylor's theorem is used to approximate the functions $f(x + F(x,z))$ and $f(x + G(x,z))$ at the point $x$ as follows:
{ \small\begin{equation}\label{taylorexpansion:f}
    \begin{aligned}
         f(x+F(x,z))=&\ 
         f(x) + \sum_{i=1}^d F_i(x,z)\int_0^1 (\partial_if)(x+\theta F(x,z))\d \theta ,\\
         f(x+G(x,z))=&\ 
         f(x) + \sum_{i=1}^d G_i(x,z)\int_0^1 (\partial_if)(x+\theta G(x,z))\d \theta.
    \end{aligned}
\end{equation}}

For every $\theta\in[0,1]$ and $z\in\mathbb{R}^d\backslash\{0\}$, define the mappings
\begin{equation*}
    \begin{aligned}
        \mathcal{T}_{F,\theta,z}:\ x\mapsto x + \theta F(x,z),\quad
        \mathcal{T}_{G,\theta,z}:\  x\mapsto x + \theta G(x,z),
    \end{aligned}
\end{equation*}
and assume both are invertible
for every $\theta\in[0,1]$ and $z\in\mathbb{R}^d\backslash\{0\}$.

Substitute \eqref{taylorexpansion:f} to \eqref{generator} and let $p_t(x_0,\cdot)$ be the probability density for $X_t$ associated with \eqref{sdeBL}, we then obtain the expression for the adjoint operator of $\mathcal{L}$ as follows, for any $f\in C^2_0(\mathbb{R}^d)$,
{\small\begin{equation*}
    \begin{aligned}
        &\int_{\mathbb{R}^d}\mathcal{L}f (x)p_t(x_0,x)\d x
        =\ -\int_{\mathbb{R}^d}\left(\sum_{i=1}^d \partial_i\left[ \left( b_i(x) - \int_{|z|<1} F_i(x,z)\nu(\d z)\right)p_t(x_0,x)\right] \right.\\
        &\ \left. -  \frac{\sigma^2}{2}\sum_{i,j=1}^d \partial_i\partial_j p_t(x_0,x) \right)f(x)\d x  - \int_{\mathbb{R}^d}\int_{|z|<1}\int_0^1\sum_{i=1}^dF_i(\mathcal{T}^{-1}_{F,\theta,z}(x),z)  (\partial_if)(x)\\
        &\ \times p_t(x_0,\mathcal{T}^{-1}_{F,\theta,z}(x))|\mathcal{J}_{F,\theta,z}(x)|\d \theta \nu(\d z)\d x - \int_{\mathbb{R}^d}\int_{|z|\geq1}\int_0^1 \sum_{i=1}^d  G_i(\mathcal{T}^{-1}_{G,\theta,z}(x),z) (\partial_if)(x)\\
        &\ \times p_t(x_0,\mathcal{T}^{-1}_{G,\theta,z}(x))|\mathcal{J}_{G,\theta,z}(x)|\d \theta \nu(\d z) \d x =:\ \int_{\mathbb{R}^d}f (x)\mathcal{L}^*p_t(x_0,x)\d x,
    \end{aligned}
\end{equation*}
}where we have used the integration by parts and the fact that $f$ vanishes at infinity, and $\mathcal{J}_{F,\theta,z}$, $\mathcal{J}_{G,\theta,z}$ are the Jacobian matrices for the inverse mappings of $\mathcal{T}_{F,\theta,z}$ and $\mathcal{T}_{G,\theta,z}$ respectively, $|\mathcal{J}_{F,\theta,z}|$, $|\mathcal{J}_{G,\theta,z}|$ denote their determinants.

We can express the corresponding L\'{e}vy--Fokker-Planck equation of \eqref{sdeBL} for any $x, y \in \mathbb{R}^d$ and $t > 0$ as follows:
{\small  \begin{equation}\label{eqn:LFPEinfinite}
    \begin{aligned}
        \frac{\partial p_t(y,x)}{\partial t} 
        =&\ -\sum_{i=1}^d\ \partial_i\left[\left(b_i(x) - \int_{|z|<1} F_i(x,z)\nu(\d z)\right)p_t(y,x) \right]
         + \frac{\sigma^2}{2} \sum_{i,j=1}^d\frac{\partial^2}{\partial x_i\partial x_j}  p_t(y,x) \\
         &\ - \int_{|z|<1}\int_0^1\sum_{i=1}^d\partial_i\left( F_i(\mathcal{T}^{-1}_{F,\theta,z}(x),z) p_t(y,\mathcal{T}^{-1}_{F,\theta,z}(x))|\mathcal{J}_{F,\theta,z}(x)| \right) \d \theta \nu(\d z)\\
        &\ - \int_{|z|\geq1}\int_0^1\sum_{i=1}^d \partial_i\left( G_i(\mathcal{T}^{-1}_{G,\theta,z}(x),z)  p_t(y,\mathcal{T}^{-1}_{G,\theta,z}(x))|\mathcal{J}_{G,\theta,z}(x)|\right)\d \theta \nu(\d z).
    \end{aligned}
\end{equation}}

We now turn to the discussion of probability flow equivalence between jump-diffusion processes and diffusion processes. By employing a similar method as in the previous section, we can establish the following results.
\begin{theorem} 
    Assuming that Assumption \ref{assum:bnu} holds, and the mappings $\mathcal{T}_{F,\theta,z}$ and $\mathcal{T}_{G,\theta,z}$ are invertible for every $\theta\in[0,1]$ and $z\in\mathbb{R}^d\backslash\{0\}$. Let $p_t$ denotes the transition density of \eqref{sdeBL}. If $p_t(y,x) > 0$ for all $t > 0$ and $y,x\in\mathbb{R}$, then the solution process of \eqref{sdeBL} shares all single marginal distributions with the solution process of the following It\^{o} diffusion SDE in the limit $\varepsilon\to0$:
   {\small \begin{equation}\label{sdeB}
       \left\{\begin{aligned}
            \d \widehat{X}^{x_0,\varepsilon}_t=&\  \left(b(\widehat{X}^{x_0,\varepsilon}_t) - \int_{|z|<1} F(\widehat{X}^{x_0,\varepsilon}_t,z)\nu(\d z) + \int_{|z|<1}\int_0^1 F(\mathcal{T}^{-1}_{F,\theta,z}(\widehat{X}^{x_0,\varepsilon}_t),z)\right. \\
            &\  \times |\mathcal{J}_{F,\theta,z}(\widehat{X}^{x_0,\varepsilon}_t)| \frac{  p^{x_0,\varepsilon}_t(\mathcal{T}^{-1}_{F,\theta,z}(\widehat{X}^{x_0,\varepsilon}_t))}{p^{x_0,\varepsilon}_t(x_0,\widehat{X}^{\varepsilon}_t)}\d \theta\ \nu(\d z)  + \int_{|z|\geq1}\int_0^1 G(\mathcal{T}^{-1}_{G,\theta,z}(\widehat{X}^{x_0,\varepsilon}_t),z)\\
            &\ \left. \times|\mathcal{J}_{G,\theta,z}(\widehat{X}^{x_0,\varepsilon}_t)| \frac{  p^{x_0,\varepsilon}_t(\mathcal{T}^{-1}_{G,\theta,z}(\widehat{X}^{x_0,\varepsilon}_t))}{p^{x_0,\varepsilon}_t(x_0,\widehat{X}^{\varepsilon}_t)}\d \theta\ \nu(\d z)  \right)\d t + \sigma \d B_t ,\quad t>0,\\
            \widehat{X}^{x_0,\varepsilon}_0\sim&\ \mathsf{N}_{x_0,\varepsilon},
       \end{aligned}\right.
    \end{equation}}and for any $A\in\mathcal{B}(\mathbb{R}^d)$ and $t\geq0$, 
    \begin{equation}
        \mathbb{P}(X_t\in A| X_0=x_0)=\ \lim_{\varepsilon\to0}\mathbb{P}(\widehat{X}^{x_0,\varepsilon}_t\in A|\widehat{X}^{x_0,\varepsilon}_0=x_0).
    \end{equation}
\end{theorem}
\begin{proof}
The proof follows a similar way as the one in Theorem \ref{theo:finite}. For the L\'{e}vy--Fokker--Planck equation \eqref{eqn:LFPEinfinite}, we rewrite it as
{\small\begin{equation}\label{eqn:FPE} 
    \begin{aligned}
        &\ \frac{\partial p_t(y,x)}{\partial t} 
        =\ \frac{\sigma^2}{2}  \frac{\partial^2}{\partial x^2 } p_t(y,x)  - \frac{\partial}{\partial x}\left[\left(b(x) - \int_{|z|<1} F(x,z)\nu(\d z)\right. \right.\\
        &\ - \int_{|z|<1}\int_0^1 \frac{ F(\mathcal{T}^{-1}_{F,\theta,z}(x),z) p_t(y,\mathcal{T}^{-1}_{F,\theta,z}(x))|\mathcal{J}_{F,\theta,z}(x)|  \d \theta\ \nu(\d z)}{p_t(y,x)}\\
         &\ \left.\left. - \int_{|z|\geq1}\int_0^1 \frac{ G(\mathcal{T}^{-1}_{G,\theta,z}(x),z) p_t(y,\mathcal{T}^{-1}_{G,\theta,z}(x))|\mathcal{J}_{G,\theta,z}(x)|  \d \theta\ \nu(\d z)}{p_t(y,x)}  \right)p_t(y,x) \right],
    \end{aligned}
\end{equation}
}which has the form of the general Fokker--Planck equation for the It\^{o} process. The ``new'' drift term 
{\small \begin{equation*}
    \begin{aligned}
        \widehat{b}^{x_0}(x,t):=&\ b(x) - \int_{|z|<1} F(x,z)\nu(\d z)  \\
        &\ + \int_{|z|>1}\int_0^1 \frac{ F(\mathcal{T}^{-1}_{F,\theta,z}(x),z) p_t(x_0,\mathcal{T}^{-1}_{F,\theta,z}(x))|\mathcal{J}_{F,\theta,z}(x)|  \d \theta \nu(\d z)}{p_t(x_0,x)}\\
        &\ + \int_{|z|\geq1}\int_0^1 \frac{ G(\mathcal{T}^{-1}_{G,\theta,z}(x),z) p_t(x_0,\mathcal{T}^{-1}_{G,\theta,z}(x))|\mathcal{J}_{G,\theta,z}(x)|  \d \theta \nu(\d z)}{p_t(x_0,x)},
    \end{aligned}
\end{equation*}}is well-defined for all $t>0$ since $p_t(x_0,\cdot)$ is smooth and positive. However, it becomes singular when $t=0$. Thus, the argument for well-posedness follows a similar approach as in Theorem \ref{theo:finite}, utilizing the concept of weak convergence. 
\end{proof}

\subsection{Short time estimate for the transition density function}
If we want to approximate the transition density of the jump-diffusion SDE \eqref{sdeBL} in short time limit, we need a result similar to Lemma \ref{lem:tp}. To the best of our knowledge, such results have only been proven for scalar cases, as illustrated in references \cite{figueroa2014small, figueroa2018small,figueroa2009small}. The remainder of this section will concentrate on the one-dimensional scenario with $F = G$.

Consider the scalar jump-diffusion process with generator \eqref{sdeBL2}. Note that the L\'{e}vy measure $\nu: \mathbb{R} \times \mathbb{R} \backslash \{0\} \to (0, \infty)$ may depend on the state variable. The corresponding L\'{e}vy--Fokker--Planck equation resembles equation \eqref{eqn:LFPEinfinite}; thus, we will omit the derivation here. To discuss the short-term estimate of the transition kernel for the jump-diffusion described by \eqref{sdeBL2}, we need to establish the following conditions. 
\begin{enumerate}[label=(C\arabic*)]
    \item \begin{enumerate}
        \item There exists a L\'{e}vy intensity $g$, which ``dominates'' the jump intensity function $\nu:\mathbb{R}\times(\mathbb{R}\backslash\{0\})\to(0,\infty)$ in the sense that $\nu(x,z)\leq g(z)$, for all $(x,z)\in\mathbb{R}\times(\mathbb{R}\backslash\{0\})$.
        \item We also assume that $\bar{\nu}(x,z):=\nu(x,z)/g(x)$ and $g$ are $C^4_b(\mathbb{R}\times[\varepsilon,\varepsilon]^c)$ and $C^4_b([\varepsilon,\varepsilon]^c)$, respectively, for any $\varepsilon>0$, and, furthermore,
        {\small\begin{equation*}
            \begin{aligned}
                \liminf_{z\to0^{\pm}}\inf_{x\in\mathbb{R}}\bar{\nu}(x,z)>0,\quad \limsup_{z\to0^{\pm}}\sup_{x\in\mathbb{R}}|z\partial_2\bar{\nu}(x,z)|<\infty,\\
                \limsup_{z\to0^{\pm}}\sup_{x\in\mathbb{R}}|z\partial^i_1\bar{\nu}(x,z)|<\infty,\quad i=0,1,2.
            \end{aligned}
        \end{equation*}}
    \end{enumerate}\label{S1}
    \item  The function $b:\mathbb{R}\to\mathbb{R}$  belongs to $C^4_b(\mathbb{R})$.\label{S2}
    \item The function $F(x,r):\mathbb{R}\times\mathbb{R}\to\mathbb{R}$ satisfies the following conditions:
    \begin{enumerate}
        \item It belongs to $C^4_b(\mathbb{R},\mathbb{R})$ and $F(x,0)=0$ for all $x\in\mathbb{R}$.
        \item For all $x,z\in\mathbb{R}$, $|\partial_2F(x,z)|>\eta$ for some constant $\eta>0$.
        \item For all $x,z\in\mathbb{R}$, $|1+\partial_1 F(x,z)|>\eta$ for some constant $\eta>0$.
    \end{enumerate}\label{S3}
    \item The L\'{e}vy density $g$ mentioned above is such that, for some $\alpha\in(0,2)$, $f(z):=g(z)|z|^{\alpha+1}$ is differentiable in $(-\varepsilon_0,0)\cup(0,\varepsilon_0)$, for some $\varepsilon_0>0$, and
    {\small\begin{equation*}
        \begin{aligned}
            \liminf_{z\to0^{\pm}} f(z)>0,\quad \limsup_{z\to0^{\pm}} f(z)<\infty,\quad \limsup_{z\to0^{\pm}} |zf'(z)|<\infty.
        \end{aligned}
    \end{equation*}}\label{S4}
\end{enumerate}
Under such conditions, we have a result analogous to Lemma \ref{lem:tp}, stated as follows:
\begin{lemma}\label{lem:tp2}(\cite[Theorem 6.1]{figueroa2018small} \& \cite[Theorem 5.2]{figueroa2014small})
Given Conditions \ref{S1}, \ref{S2}, \ref{S3} and \ref{S4}, let $p_t(y, x)$ denote the transition kernel of \eqref{sdeBL2}. Then,
{\small\begin{equation*}
\lim_{t \downarrow 0} \frac{p_t(y, x)}{t} = -\frac{\partial}{\partial (x-y)}\left(\int_{\{z : F(y, z) \geq  x-y\}}\nu(y,z)\d z\right)=:\nu_F(y,x-y),\quad x\neq y.
\end{equation*}}
\end{lemma}

\subsection{Discrete Onsager--Machlup functional for jump-diffusion processes with infinite jump activity}\label{subsection:discreteOM}
Under Lemma \ref{lem:tp2}, we know that the short time estimate of $p_t$ is determined by the function $\nu_F$. However, for each $x \in \mathbb{R}^d$, $\nu_F(x,\cdot)$ is undefined at 0. Consequently, it is infeasible to derive a continuous-time OM functional for jump-diffusions with infinite jump activity; however, a discrete version remains achievable. The finite-dimensional distribution for the jump-diffusion \eqref{sdeBL2} can be determined using a similar approach as previously outlined.
{\small   \begin{equation}\label{PI:finite3} 
    \begin{aligned}
      &\mathbb{P}(X_{t_i}\in A_i,\ i=1,\cdots,n) \\
      =&\ \int_{\substack{\psi\in C_{x_0,x_T[0,T]} \\ \psi_{t_i}\in A_i,i=1,\cdots,n}}\lim_{\varepsilon\to0}\exp\left\{ - \sum_{i=1}^n\int_{t_{i-1}}^{t_i} \left( \frac{1}{2\sigma^2}\bigg|\dot{\psi}^{(i)}_s   - b(\psi^{(i)}_s) + \int_{|z|<1} F(\psi^{(i)}_s,z)\nu(\psi^{(i)}_s,\d z) \right.\right.\\
        &\ - \int_0^1\int_{\mathbb{R}\backslash\{0\}} F(\mathcal{T}^{-1}_{F,\theta,z}(\psi^{(i)}_s),z)|\mathcal{J}_{F,\theta,z}(\psi^{(i)}_s)| \frac{p_s^{\psi_{t_{i-1}},\varepsilon}(\mathcal{T}^{-1}_{F,\theta,z}(\psi^{(i)}_s) )  }{ p_s^{\psi_{t_{i-1}},\varepsilon}(\psi^{(i)}_s )} \nu(\mathcal{T}^{-1}_{F,\theta,z}(\psi^{(i)}_s),\d z)\ \d\theta    \bigg|^2   \\
        &\   + \frac{1}{2} \nabla\cdot \left[ b(\psi^{(i)}_s) -\int_{|z|<1} F(\psi^{(i)}_s,z)\nu(\psi^{(i)}_s,\d z)  + \int_0^1 \int_{\mathbb{R}\backslash\{0\}} F(\mathcal{T}^{-1}_{F,\theta,z}(\psi^{(i)}_s),z)|\mathcal{J}_{F,\theta,z}(\psi^{(i)}_s)| \right. \\
        &\ \left.\left.\left. \times \frac{p_s^{\psi_{t_{i-1}},\varepsilon}(\mathcal{T}^{-1}_{F,\theta,z}(\psi^{(i)}_s) )  }{ p_s^{\psi_{t_{i-1}},\varepsilon}(\psi^{(i)}_s )} \nu(\mathcal{T}^{-1}_{F,\theta,z}(\psi^{(i)}_s),\d z)\ \d\theta  \right] \right)\d s \right\}\mathcal{D}(\psi)\\
        =&\ \int_{\substack{\psi\in C_{x_0,x_T[0,T]} \\ \psi_{t_i}\in A_i,i=1,\cdots,n}}\exp\left\{ - \sum_{i=1}^n\int_{t_{i-1}}^{t_i} \left( \frac{1}{2\sigma^2}\bigg|\dot{\psi}^{(i)}_s   - b(\psi^{(i)}_s) +   \int_{|z|<1} F(\psi^{(i)}_s,z)\nu(\psi^{(i)}_s,\d z) \right.\right.\\
        &\ -\int_0^1\int_{\mathbb{R}\backslash\{0\}} F(\mathcal{T}^{-1}_{F,\theta,z}(\psi^{(i)}_s),z)|\mathcal{J}_{F,\theta,z}(\psi^{(i)}_s)|\frac{\nu_F(\psi_{t_{i-1}},\mathcal{T}^{-1}_{F,\theta,z}(\psi^{(i)}_s) - \psi_{t_{i-1}} )  }{ \nu_F(\psi_{t_{i-1}},\psi^{(i)}_s  -\psi_{t_{i-1}} )}\\
        &\ \times \nu(\mathcal{T}^{-1}_{F,\theta,z}(\psi^{(i)}_s),\d z)\ \d\theta    \bigg|^2  + \frac{1}{2} \nabla\cdot \left[ b(\psi^{(i)}_s)  -  \int_{|z|<1} F(\psi^{(i)}_s,z)\nu(\psi^{(i)}_s,\d z)  \right.\\
        &\ \left.\left.\left.  + \int_0^1 \int_{\mathbb{R}\backslash\{0\}} F(\mathcal{T}^{-1}_{F,\theta,z}(\psi^{(i)}_s),z)|\mathcal{J}_{F,\theta,z}(\psi^{(i)}_s)|  \frac{\nu_F(\psi_{t_{i-1}},\mathcal{T}^{-1}_{F,\theta,z}(\psi^{(i)}_s)-\psi_{t_{i-1}} )  }{ \nu_F(\psi_{t_{i-1}},\psi^{(i)}_s -\psi_{t_{i-1}})}  \right] \right.\right.\\
        &\  \times \d\nu(\mathcal{T}^{-1}_{F,\theta,z}(\psi^{(i)}_s),z)\ \d\theta + \mathcal{O}(T/n)\bigg)\d s \bigg\}\mathcal{D}(\psi).
    \end{aligned}
\end{equation}}Note that the term $\nu_F(\psi_{t_{i-1}},\psi^{(i)}_s-\psi_{t_{i-1}})$ formally tends to $\nu_F(\psi_{t_{i-1}},0)$ as $s\to t_{i-1}$ when $n\to\infty$. Thus we cannot obtain a continuous time OM functional when $n\to\infty$ since $\nu_F(\psi_{t_{i-1}},0)$ is not defined, as discussed as before.

Recall that velocities can be approximated by the differences between positions. Let us now approximate \eqref{PI:finite3} as follows:
{\small \begin{equation}\label{PI2}
    \begin{aligned}
        &\ \mathbb{P}(X_{t_i}\in A_i,\ i=1,\cdots,n)\\
        =&\ \int_{\prod_{i=1}^{n-1}A_i}  \exp\left\{ - \sum_{i=1}^n  \frac{1}{2\sigma^2}\left| \frac{x_{i}-x_{i-1}}{\Delta t} - b(x_i) + \int_{|z|<1}F(x_i,z)\nu(x_i,\d z) \right. \right. \\
        &\  \left. - \int_0^1\int_{\mathbb{R}\backslash\{0\}} F(\mathcal{T}^{-1}_{F,\theta,z}(x_i),z )|\mathcal{J}_{F,\theta,z}(x_i)|\frac{\nu_F(x_{i-1},\mathcal{T}^{-1}_{F,\theta,z}(x_i) - x_{i-1} )}{\nu_F(x_{i-1},x_i-x_{i-1})} \right. \\
        &\ \times \nu(\mathcal{T}^{-1}_{F,\theta,z}(x_i),\d z)\ \d\theta  \bigg|^2 \Delta t + \frac{1}{2}\sum_{i=1}^n\Delta t \frac{\d}{\d x}\left(b(x) - \int_{|z|<1}F(x,z)\nu(x,\d z)  \right. \\
        &\ + \int_0^1\int_{\mathbb{R}\backslash\{0\}} F(\mathcal{T}^{-1}_{F,\theta,z}(x),z )|\mathcal{J}_{F,\theta,z}(x)|\frac{\nu_F(x_{i-1},\mathcal{T}^{-1}_{F,\theta,z}(x)- x_{i-1} )}{\nu_F(x_{i-1},x - x_{i-1})} \\
        &\  \times \nu(\mathcal{T}^{-1}_{F,\theta,z}(x),\d z)\ \d\theta \bigg)\bigg|_{x=x_i}  + \mathcal{O}(\Delta t) \bigg\} \left(\frac{1}{\sqrt{(2\pi)\Delta t}}\right)^n\d x_1\cdots\d x_{n-1},
    \end{aligned}
\end{equation}}
where we use the Euler method for process \eqref{sdeBL2} with right endpoint as
{ \begin{equation*}
    \begin{aligned}
         x_i=&\ x_{i-1} + \left( b(x_i) - \int_{|z|<1}F(x_i,z)\nu(\d z) - \int_0^1\int_{\mathbb{R}\backslash\{0\}} F(\mathcal{T}^{-1}_{F,\theta,z}(x_i),z )|\mathcal{J}_{F,\theta,z}(x_i)| \right. \\
        &\ \left.  \times \frac{\nu_F(x_{i-1},\mathcal{T}^{-1}_{F,\theta,z}(x_i) -x_{i-1})}{\nu_F(x_{i-1},x_i - x_{i-1})} \nu(\mathcal{T}^{-1}_{F,\theta,z}(x_i),\d z)\ \d\theta \right)\Delta t  + \sigma(B_i-B_{i-1}).
    \end{aligned}
\end{equation*}}Now, \eqref{PI2} gives us a path integrals like representations for a ``discrete path'' $\{x_i\}_{i=0}^n$ as given by \eqref{dOM}. We summarize the above results in the following theorem. \begin{theorem}\label{theo:infinite}
    For the jump-diffusion SDE \eqref{sdeBL} with $F=G$ and $d=1$, assuming that Assumption \ref{assum:bnu} and Conditions \ref{S1}, \ref{S2}, \ref{S3} and \ref{S4} are satisfied, for any collection of sets $\{A_i,i=1,\cdots,n\}$ and a time partition $t_1\leq t_2\leq\cdots\leq t_n$, it follows that
    \begin{equation*} 
    \begin{aligned}
        &\ \mathbb{P}(X_{t_i}\in A_i,\ i=1,\cdots,n)\\
        =&\ \int_{\prod_{i=1}^{n-1}A_i}  \exp\left\{ - \sum_{i=1}^n  \frac{1}{2\sigma^2}\left| \frac{x_{i}-x_{i-1}}{\Delta t} - b(x_i) + \int_{|z|<1}F(x_i,z)\nu(x_i,\d z) \right. \right. \\
        &\  \left. - \int_0^1\int_{\mathbb{R}\backslash\{0\}} F(\mathcal{T}^{-1}_{F,\theta,z}(x_i),z )|\mathcal{J}_{F,\theta,z}(x_i)|\frac{\nu_F(x_{i-1},\mathcal{T}^{-1}_{F,\theta,z}(x_i) - x_{i-1} )}{\nu_F(x_{i-1},x_i-x_{i-1})} \right. \\
        &\ \times \nu(\mathcal{T}^{-1}_{F,\theta,z}(x_i),\d z)\ \d\theta  \bigg|^2 \Delta t + \frac{1}{2}\sum_{i=1}^n\Delta t \frac{\d}{\d x}\left(b(x) - \int_{|z|<1}F(x,z)\nu(x,\d z)  \right. \\
        &\ + \int_0^1\int_{\mathbb{R}\backslash\{0\}} F(\mathcal{T}^{-1}_{F,\theta,z}(x),z )|\mathcal{J}_{F,\theta,z}(x)|\frac{\nu_F(x_{i-1},\mathcal{T}^{-1}_{F,\theta,z}(x)- x_{i-1} )}{\nu_F(x_{i-1},x - x_{i-1})} \\
        &\  \times \nu(\mathcal{T}^{-1}_{F,\theta,z}(x),\d z)\ \d\theta \bigg)\bigg|_{x=x_i}  + \mathcal{O}(\Delta t) \bigg\} \left(\frac{1}{\sqrt{(2\pi)\Delta t}}\right)^n\d x_1\cdots\d x_{n-1}.
    \end{aligned}
\end{equation*}
\end{theorem}

\begin{remark}

We compare our derived Onsager--Machlup (OM) functional, as indicated in \eqref{dOM}, with the OM functional outlined by \cite{chao2019onsager}. In their analysis, the function $F(x,z)$ is simplified to $z$, the Lévy measure $\nu$ does not depend on the state variable $x$, and the dimension is $d=1$. If we neglect the non-local integral component of \eqref{dOM} specified by: 
\begin{equation}\label{nonlocal}
    \int_0^1\int_{\mathbb{R}\backslash\{0\}} F(\mathcal{T}^{-1}_{F,\theta,z}(x_i),z )|\mathcal{J}_{F,\theta,z}(x_i)|\frac{\nu_F(x_{i-1},\mathcal{T}^{-1}_{F,\theta,z}(x_i) -x_{i-1})}{\nu_F(x_{i-1},x_i-x_{i-1})} \nu(\mathcal{T}^{-1}_{F,\theta,z}(x_i),\d z)\ \d\theta ,
\end{equation}
and consider a smooth path $\phi$, with $x_i = \phi_{t_i}$ for $0=t_0<t_1<\cdots<t_n=T$, the other terms of $S_X^{\mathrm{dOM}}$ converge to
{\small\begin{equation*}
    \begin{aligned}
        &\sum_{i=1}^n  \frac{1}{2\sigma^2}\left| \frac{x_{i}-x_{i-1}}{\Delta t} - b(x_i) + \int_{|z|<1}z\nu(\d z)   \right|^2 \Delta t + \frac{1}{2}\sum_{i=1}^n b'(x_i) \Delta t\\
       \to&\ \frac{1}{2}\int_0^T \left( \frac{ \left| \phi^2_s-b(\phi_s) + \int_{|z|<1}z\nu(\d z)\right|^2}{\sigma^2} + b'(\phi_s)\right)\d s,\quad n\to\infty,
    \end{aligned}
\end{equation*}
}which matches the OM functional derived in \cite[equation (4.25)]{chao2019onsager}, albeit with a constant difference.

However, the non-local integral in \eqref{nonlocal} plays a crucial role in accurately evaluating the probability of paths within a specified tube set, thereby essential for deriving a complete OM functional. The omission of this integral in \cite[equation (4.25)]{chao2019onsager} occurs because the primary approach used involves applying the Girsanov theorem to evaluate the probability of paths within a tube set after changing the probability measure from $\mu_X$ (induced by the jump-diffusion) to $\mu_{W^c}$ induced by $W^c:=\sigma W+\int_0^t\int_{|z|<1}z\mathcal{N}(\d s,\d z)$, as shown in equation (4.4) therein. This probability is then considered as the sum of all paths around a reference tube, involving a Poisson integral related to the non-local integral \eqref{nonlocal}, as elaborated in \cite[Section 4 \& Appendix]{chao2019onsager}. Nonetheless, the Poisson integrals for trajectories within the tube set around the reference path were overlooked in \cite[equation (4.25)]{chao2019onsager}. This omission is significant because  in any tube set, there is no uniform upper limit on the number of jumps for all sample trajectories of the L\'{e}vy process. This means the Onsager--Machlup functional as presented in \cite[equation (4.25)]{chao2019onsager} offers an approximation, providing an incomplete characterization of the paths in jump-diffusion processes with infinite jump activity.
\end{remark}
 
From the discussion above, it is evident that for jump-diffusion processes with infinite jump activity, although the corresponding Onsager--Machlup functional cannot be realized in a continuous-time format, its discrete form, as indicated in \eqref{dOM}, remains valid for estimating the finite-dimensional distribution of the stochastic system, as shown in \eqref{PI2}.

\section{Conclusion}\label{discussion}
We have successfully derived the Onsager--Machlup functional for jump-diffusion processes with finite jump activity. The key tool of our work is the probability flow between the Levy-Fokker-Planck equation and the Fokker-Planck equation for diffusion process, which makes it possible to reactivate the established methods of path integrals and the Girsanov theorem. This connection provides a new and powerful strategy for analysing non-Gaussian stochastic systems by exploiting the results for Gaussian systems.

For the jump-diffusion SDE  with  infinite jump activity, we have formulated a discrete variant of the Onsager--Machlup functional on any finite number of discrete time partitions. This  discrete version effectively captures the behavior of these complex stochastic systems and also facilitates the numerical calculation of the minimum  action paths associated with the   Onsager--Machlup functional. 

\bibliographystyle{siamplain}
\bibliography{references}

\end{document}